\documentclass{article}

\title{Tight Convergence Bounds for the Classical Kaczmarz Method}
\author{Runbo Yu\\
University of Wisconsin-Madison\\
\texttt{ryu86@wisc.edu}
\and
Jelena Diakonikolas\\
University of Wisconsin-Madison\\
\texttt{jelena@cs.wisc.edu}
}
\date{}

\usepackage[margin=1in]{geometry}
 \usepackage{times}
 \usepackage{amsmath,amsfonts,amsthm,bm}

\usepackage{dsfont}
\usepackage{mathtools}
\usepackage{cite}

\usepackage[colorlinks,urlcolor=blue,citecolor=blue,linkcolor=blue]{hyperref}
 \usepackage{color}

\usepackage{cleveref}

\usepackage{thm-restate}

\usepackage{algorithm}
\usepackage[noend]{algpseudocode}

\usepackage{graphicx}
\usepackage{subcaption}

\newtheorem{theorem}{Theorem}
\newtheorem{lemma}{Lemma}
\newtheorem{corollary}{Corollary}
\newtheorem{remark}{Remark}

\newtheorem{claim}{Claim}

\def\1{\bm{1}}

\def\vzero{{\bm{0}}}

\def\va{{\bm{a}}}
\def\vb{{\bm{b}}}

\def\vd{{\bm{d}}}
\def\ve{{\bm{e}}}

\def\vi{{\bm{i}}}

\def\vo{{\bm{o}}}

\def\vs{{\bm{s}}}

\def\vu{{\bm{u}}}
\def\vv{{\bm{v}}}
\def\vw{{\bm{w}}}
\def\vx{{\bm{x}}}

\def\mA{{\bm{A}}}

\def\mG{{\bm{G}}}

\def\mI{{\bm{I}}}

\def\mL{{\bm{L}}}
\def\mM{{\bm{M}}}

\def\mP{{\bm{P}}}

\def\mU{{\bm{U}}}
\def\mV{{\bm{V}}}

\def\mSigma{{\bm{\Sigma}}}

\newcommand\norm[1]{\left\| #1 \right\|}

\def\sR{{\mathbb{R}}}

\newcommand{\E}{\mathbb{E}}

\newcommand{\R}{\mathbb{R}}

\newcommand{\range}{\mathrm{range}}
\newcommand{\rank}{\mathrm{rank}}
\newcommand{\op}{\mathrm{op}}

\DeclareMathOperator*{\argmin}{arg\,min}

\newcommand{\innp}[1]{\langle #1 \rangle}
\usepackage[colorinlistoftodos]{todonotes}

\newcommand{\diag}{\mathrm{diag}}

\begin{document}

\maketitle

\begin{abstract}
    The classical method of Kaczmarz, introduced in 1937, is a  textbook iterative method for solving linear systems $\mA \vx = \vb.$ Despite its widespread use, particularly in the context of solving inverse problems where it is commonly included in software packages within Matlab, Python, and Julia, the precise characterization of convergence has long been deemed difficult to obtain. While different bounds on convergence rates have been established, they are largely unsatisfying as they cannot explain the classical, cyclic-update method's efficient convergence in practice. In this work, we obtain a tight characterization of convergence of the classical Kaczmarz method, establishing both linear and sublinear convergence bounds. The worst-case tight (i.e., exactly attained by some instances in the considered family) bounds are expressed in terms of a fixed matrix that depends only on $\mA,$ but are not fully interpretable in terms of the matrix spectrum and row correlations, which had been observed to have an impact on convergence. We thus provide relaxations of these bounds that are fully expressible in terms of matrix row norms, row correlations, rank, and extremal positive singular values. The provided relaxed bounds explain one-cycle convergence in special cases where the matrix rows are all either parallel or orthogonal to each other. We further argue that the dependence on different parameters appearing in the bounds is necessary and within a small constant factor of the best attainable in the worst case. Finally, our bounds explain why the classical cyclic update is faster than the randomized one when matrix rows are weakly correlated, which is often observed in inverse problems where the cyclic method is used.   
\end{abstract}

\section{Introduction}

The classical method of Kaczmarz is a simple iterative method for solving linear systems described by 
\begin{equation}\label{eq:main-prob}\tag{P}
    \mA \vx = \vb,
\end{equation}
where $\mA \in \R^{m \times n}$, $\vx \in \R^n$, and $\vb \in \R^m$, with $m\ge 1, n \geq 1$. 
Denote the rows of $\mA$ by $\va_i$, $i \in [m]$. Starting from $\vx_0$, the Kaczmarz method updates its iterate based on one row (i.e., one linear equality) at a time, traversing all rows in a cyclic manner from $1$ to $m$\footnote{Setting the ordering to be $1, 2, \dots, m$ is without loss of generality, as any deterministic ordering can be covered by reordering the rows of the linear system. For simplicity, throughout the paper, we take the ordering to be $1, 2, \dots, m$, as described in the classical method \cite{kaczmarz1937angenaherte}.}, and sets $\vx_{k+1} = \vx_{k+1,0} = \vx_{k,m}$ (where $\vx_{0, 0} = \vx_0$) at the end of each cycle. Denoting $[m] := \{1, \dots, m\}$ and by $\vb^{(i)}$ the $i^{\mathrm{th}}$ entry of $\vb,$ each within-cycle update is given by 
\begin{equation}\label{eq:Kaczmarz-update-within-cycle}
    \vx_{k,i} =\vx_{k,i-1} - \frac{\innp{\va_i, \vx_{k,i-1}}-\vb^{(i)}}{\|\va_i\|^2} \va_i,\; i \in [m].
\end{equation}
Thus, a full cycle update is given by 
\begin{equation}\label{eq:Kaczmarz-update-full}
    \vx_{k+1} =\vx_{k} - \sum_{i=1}^m\frac{\innp{\va_i, \vx_{k,i-1}}-\vb^{(i)}}{\|\va_i\|^2} \va_i.
\end{equation}

Kaczmarz method is widely used in practice, particularly in the context of solving inverse problems. For instance, it is used in computed tomography (where it is also known as the algebraic reconstruction technique---ART) \cite{gordon1970algebraic,hansen2018air,natterer2001mathematics}, electron tomography \cite{dahmen2016ettention}, magnetic particle imaging \cite{kluth2019enhanced}, and in numerical solvers for partial differential equations \cite{brannick2014bootstrap}. As noted in \cite{Vershynin_2007}, this method was implemented in the first medical scanner \cite{hounsfield1973computerized}. Classical, cyclic version of the method is implemented in various software packages; for instance, in AIR Tools II (MATLAB) \cite{hansen2018air}, ASTRA (Python/MATLAB) \cite{van2015astra,van2016fast}, and MPIReco.jl (Julia) \cite{knopp2019mpireco}, while different variants of the classical method are also implemented as part of TomoPy (Python) \cite{gursoy2014tomopy}, scikit-image (Python) \cite{van2014scikit}, and hypre (C/C++) \cite{falgout2002hypre} software packages. 

While simple to state and implement, Kaczmarz method has nevertheless long resisted any meaningful convergence guarantees. In fact, the difficulty of analyzing the classical cyclic version of the method is what motivated the introduction of the randomized version, where the ordering of updates \eqref{eq:Kaczmarz-update-within-cycle} is determined by sampling rows $i$ at random, according to a fixed probability distribution \cite{Vershynin_2007}. As the authors of \cite{Vershynin_2007} note:
\begin{center}
    \emph{``useful
theoretical estimates of the rate of convergence of the Kaczmarz method [...] are difficult to
obtain, at least for $m > 2.$''}
\end{center}
A similar sentiment is also echoed in more recent work; for instance, \cite{steinerberger2021randomized} remarks 
\begin{center}
\emph{``given
the intricate underlying geometry, convergence rates are diﬃcult to obtain.''}  
\end{center}
As a result, much of the theoretical development for Kaczmarz method and its extensions has focused on randomized variants \cite{steinerberger2021randomized,derezinski2025randomized,tan2019phase,huang2022linear,needell2014paved,gower2015randomized,attia2026fast}. On the other hand, theoretical results for the classical method have been more pessimistic \cite{sun2021worst,evron2022catastrophic,Oswald_2015,Dai_2015}, leading to bounds that (often polynomially) depend on the matrix dimensions or rank, and do not explain the method's efficient convergence frequently observed in practice, particularly in the context of row ordering and correlations among the matrix rows.  

In this work, we provide sharp convergence characterizations of the classical Kaczmarz method \eqref{eq:Kaczmarz-update-full}. Before summarizing our contributions, we introduce the necessary notation and state basic facts and assumptions used.

\subsection{Preliminaries}\label{sec:prelims}

As noted earlier, throughout the paper, we assume we are given a matrix $\mA \in \sR^{m \times n}$ and a vector $\vb \in \sR^m,$ where $m \geq 1, n \geq 1$, and $\vb \in \range(\mA),$ so the linear system has at least one solution $\vx^* \in \sR^n$, $\mA \vx^* = \vb$. We assume the system contains no zero rows (which can be removed without loss of generality). Throughout the paper, if the solution to the linear system is not unique, given an initial vector $\vx_0,$ we fix $\vx^*$ to be the solution that minimizes $\norm{\vx_0 - \vx^*}:$ 
\begin{equation}\label{eq:x^*-specification}
    \vx^* = \argmin_{\vx: \mA\vx = \vb}\norm{\vx_0 - \vx} = \vx_0 - \mA^\dagger(\mA\vx_0 - \vb),
\end{equation}
where $\mA^\dagger$ is the Moore-Penrose pseudoinverse of $\mA.$ 
This choice of $\vx^*$ ensures that $\vx_0 - \vx^* \in \range(\mA^\top)$. Because  Kaczmarz update \eqref{eq:Kaczmarz-update-within-cycle} adds a multiple of a row of \(\mA\), we also have $\vx_{k, i} - \vx^* \in \range(\mA^\top)$ for all $k, i$; in other words, all iterates shifted by $\vx^*$ remain in $\range(\mA^\top).$

We let $\mA$ be of arbitrary rank $r := \rank(\mA) \le \min\{m, n\}.$ Using the singular value decomposition, we write $\mA = \mU \mSigma \mV^\top$, where $\mU \in \R^{m \times r}$ is the matrix with left singular vectors as its columns, $\mV \in \R^{n \times r}$ is the matrix with right singular vectors as its columns, and $\mSigma = \diag(\sigma_1, \cdots, \sigma_r)$ with $\sigma_1 \ge ... \ge \sigma_r > 0$ are the (positive) singular values of $\mA$. Further, $\mU^\top\mU = \mV^\top\mV = \mI_r,$ where $\mI_r \in \sR^{r\times r}$ denotes the identity matrix, and the columns of $\mV$ span the range of $\mA^\top$. We denote the columns of $\mV$ by $\vv_\ell$, $\ell \in [r],$ where $[r]:=\{1, \dots, r\}$. We use $\hat{\va}_i = \frac{\va_i}{\|\va_i\|}$, where $\norm{\cdot} = \norm{\cdot}_2,$ to denote the rows of $\mA$ normalized to unit norm, and let $\hat{\mA}$ be the corresponding row-normalized matrix. Note that because row normalization preserves the row space, the columns of $\mV$ span $\range(\hat{\mA}^\top).$ Throughout, $\norm{\mA}_\op$ denotes the operator norm of $\mA$ and $\norm{\mA}_F$ denotes its Frobenius norm. 

Using that $\mA \vx^* = \vb$, Kaczmarz update \eqref{eq:Kaczmarz-update-within-cycle} can  equivalently be written in terms of normalized row vectors as
\begin{equation}\label{eq:Kacmarz-update-normalized}
    \vx_{k, i} = \vx_{k, i-1} - \innp{\hat{\va}_i, \vx_{k,i-1}-\vx^*} \hat{\va}_i. 
\end{equation}
Inspired by \cite{steinerberger2021randomized}, our analysis tracks $\ve_k^{(\ell)}=\innp{\vx_k-\vx^*, \vv_\ell}$, the $\ell^\mathrm{th}$ component of $\ve_k := \mV^\top (\vx_k-\vx^*)$, which is the projection of the solution residual $\vx_k-\vx^*$ onto the $\ell^{\mathrm{th}}$ right singular vector. Our choice of $\vx^*$ in \eqref{eq:x^*-specification} ensures that
\begin{equation}\label{eq:needed-gen-V}
    \vx_k - \vx^* = \mV\ve_k \quad \text{ and } \quad \norm{\vx_k - \vx^*} = \norm{\ve_k}.
\end{equation}
Our tight convergence results are stated in terms of the following matrices:
\begin{equation}\label{eq:L-and-M-defs}
    \mL :=
\begin{pmatrix}
1 
& 0 
& 0 
& \cdots 
& 0 \\

\innp{\widehat{\va}_2,\widehat{\va}_1}
& 1
& 0
& \cdots
& 0 \\

\innp{\widehat{\va}_3,\widehat{\va}_1}
& \innp{\widehat{\va}_3,\widehat{\va}_2}
& 1
& \cdots
& 0 \\

\vdots
& \vdots
& \vdots
& \ddots
& \vdots \\
\innp{\widehat{\va}_m,\widehat{\va}_1}
& \innp{\widehat{\va}_m,\widehat{\va}_2}
& \innp{\widehat{\va}_m,\widehat{\va}_3}
& \cdots
& 1
\end{pmatrix}, \quad \mM := \mI_r - (\hat{\mA}\mV)^\top \mL^{-1}\hat{\mA}\mV,
\end{equation}
where $\mI_r$ denotes the $r \times r$ identity matrix. 
By convention, in our results, $ab/(a+b) = 0$ for $a = b = 0.$

\subsection{Contributions}

As stated earlier, our main contributions are sharp convergence characterizations of Kaczmarz method, recovering known special cases for which the method converges in a single cycle, and aligning with prior empirical observations regarding the method's convergence depending strongly on the correlations among the pairs of matrix rows. Importantly, the provided bounds are dependent on the ordering of the rows, which is known to affect the method's performance \cite{Oswald_2015}.  

Our analysis begins by deriving the identity $\ve_{k+1} = \mM\ve_k$, which applies to arbitrary initial vectors and all $k \geq 0$ (\Cref{lem:dynamics-of-ek}, main technical lemma). Further, it is not hard to argue that $\norm{\mM}_\op < 1$, which immediately implies that $\ve_k = \mM^k\ve_0$ diminishes geometrically. Starting with this basic identity, we then establish our main result:
\begin{theorem}[Main Theorem]\label{thm:matrix-power-bound}
    Given a consistent linear system \eqref{eq:main-prob} where $\rank(\mA) = r,$ let $\mL, \mM$ be defined by \eqref{eq:L-and-M-defs}. Let Kaczmarz method be initialized at an arbitrary $\vx_0$ and let $\vx^*$ satisfy \eqref{eq:x^*-specification}. Then, for every $k\ge 0$, 
    \begin{gather}
        \sup_{\vx_0 \neq \vx^*}\frac{\norm{\vx_k - \vx^*}}{\norm{\vx_0 - \vx^*}} =\norm{\mM^k}_\op, \quad \sup_{\vx_0 \neq \vx^*}\frac{\norm{\mA\vx_k - \vb}}{\norm{\vx_0 - \vx^*}} = \norm{\mSigma \mM^k}_\op, \; \text{ and } \label{eq:tight-contraction}\\
        \frac{1}{k+1}\sum_{i=0}^k \norm{\mL^{-1}\hat{\mA}(\vx_i - \vx^*)}^2 = \frac{\norm{\vx_0 - \vx^*}^2 - \norm{\vx_{k+1} - \vx^*}^2}{k+1}. \label{eq:tight-telescoping}
    \end{gather}
    As a consequence, the squared residual satisfies, for all $k \geq 1,$ 
    \begin{gather*}
        \| \mA \vx_k - \vb\|^2 \le \min\Big\{\norm{\mSigma \mM}_\op^2\|\mM^{k-1}\|_\op^2 \|\vx_0 -\vx^* \|^2,\, 
        \frac{ (r -1)\max_{i \in [m]}\norm{\va_i}^2\norm{\mL - \mI_m}_\op^2 \norm{\vx_0 - \vx^*}^2}{k}\Big\}, \; \text{and}\\
        \frac{1}{k}\sum_{i=1}^k \|\mA(\vx_i -  \vx^*)\|^2  \leq  \frac{ \max_{i \in [m]}\norm{\va_i}^2\norm{\mL - \mI_m}_\op^2 \norm{\vx_0 - \vx^*}^2}{k}.
    \end{gather*}
\end{theorem}
While tight, the bounds \eqref{eq:tight-contraction}, \eqref{eq:tight-telescoping} stated in \Cref{thm:matrix-power-bound} are not fully interpretable in terms of geometric properties of the problem (such as rank, spectrum, and row correlations). The relaxation for the squared residual in the second part of the theorem gives bounds for the last and the average iterate that are easier to interpret, by observing that $\norm{\mL - \mI_m}_\op^2 \leq \norm{\mL - \mI_m}_F^2 = \sum_{i < j}\innp{\hat{\va}_i, \hat{\va}_j}^2$ and that $\innp{\hat{\va}_i, \hat{\va}_j} = \cos(\angle(\va_i, \va_j))$ are precisely the row correlations. In particular, these bounds immediately imply that the right-hand side is zero (and thus Kaczmarz method converges in a single cycle) whenever the matrix rows are either orthogonal to each other ($\innp{\hat{\va}_i, \hat{\va}_j} = 0$ for $i \neq j,$ implying $\norm{\mL - \mI}_\op = 0$) or parallel to each other (since in that case $r - 1 = 0$). 

The bound $\| \mA \vx_k - \vb\|^2 \le \norm{\mSigma \mM}_\op^2\|\mM^{k-1}\|_\op^2 \|\vx_0 -\vx^* \|^2$, however, is less easily interpretable. To obtain more interpretable bounds, we use matrix submultiplicativity to bound $\|\mM^{k-1}\|_\op^2 \leq \|\mM\|_\op^{2(k-1)},$ which immediately translates into linear rate of convergence, since $\norm{\mM}_\op < 1.$ We then provide more explicit bounds on both the prefactor $\norm{\mSigma \mM}_\op^2$ (in \Cref{lem:spectrum-of-SigmaM}) and the contraction factor $\norm{\mM}_\op$ (in \Cref{lem:spectrum-of-M}), dependent only on (subsets of) row norms, rank, row correlations, and extreme positive singular values $\sigma_1(\mA), \sigma_r(\mA),$ with the resulting convergence bound summarized in \Cref{cor:row-correlation-bound}. The relaxation for the prefactor obtained in \Cref{lem:spectrum-of-SigmaM} is tight; we provide a $2 \times 2$ example in \Cref{sec:tightness} for which the obtained bounds hold with equality. The bound on the contraction factor $\norm{\mM}_\op$ in \Cref{lem:spectrum-of-M} is of the form
\[\
    \|\mM\|_\op^2 \le\frac{\| \mL - \mI_m\|_\op^2}{\| \mL - \mI_m\|_\op^2+\sigma_r^2(\hat{\mA})} = 1 - \frac{\sigma_r^2(\hat{\mA})}{\| \mL - \mI_m\|_\op^2 + \sigma_r^2(\hat{\mA})}. 
    \]
Finally, as noted above, $\norm{\mL - \mI_m}_\op^2 \leq  \sum_{i < j}\innp{\hat{\va}_i, \hat{\va}_j}^2$, and we additionally use the ideas from \cite{Oswald_2015} to argue that $\norm{\mL - \mI_m}_\op \leq (1 + \lceil \log_2(r) \rceil/2)\sigma_1^2(\hat{\mA}).$ These bounds imply fast linear convergence of the method when either the rows are weakly correlated (as measured by $\sum_{i < j}\innp{\hat{\va}_i, \hat{\va}_j}^2$) relative to the minimum positive singular value $\sigma_r^2(\hat{\mA})$ of the row-normalized matrix $\hat{\mA}$ or when the rank $r$ is low and the row-normalized matrix $\hat{\mA}$ is well conditioned (as measured by $\sigma_1^4(\hat{\mA})/\sigma_r^2(\hat{\mA})$). Note here that $\sigma_i(\hat{\mA}) \in \big[\frac{\sigma_i(\mA)}{\max_{i \in [m]}\norm{\va_i}}, \, \frac{\sigma_i(\mA)}{\min_{i \in [m]}\norm{\va_i}}\big]$. 

As our final contribution, in \Cref{sec:tightness}, we provide simple examples that imply worst-case constant-factor tightness of provided (linear and sublinear) convergence bounds.

\subsection{Related Work}\label{sec:related-work}

That the Kaczmarz method should (asymptotically) converge to a solution is well established; this was, in fact, noted as ``obvious from a geometric perspective'' (but without a proof) in the original paper \cite{kaczmarz1937angenaherte}. Formally, the asymptotic convergence of the iterates of Kaczmarz method to the solution closest to the initial iterate was established in \cite{tanabe1971projection}. Our focus in this work is on obtaining sharp quantitative estimates characterizing convergence of the method, with primary focus on bounding the squared residual (though the results also imply quantitative convergence bounds for the iterate convergence; see \Cref{rem:convergence-of-the-iterates}).

In terms of non-asymptotic convergence results for the classical Kaczmarz method, there are various existing results in the literature, often considering generalizations of the method via either more general projections onto convex sets (POCS) or relaxations $\vx_{k,i}=\vx_{k,i-1}-\lambda\frac{\innp{\va_i,\vx_{k,i-1}}-b_i}{\|\va_i\|^2}\va_i$, for $\lambda \in (0, 2)$ (so the standard method is recovered for $\lambda = 1$). To provide an apples-to-apples comparison, we state the results from prior work that specifically apply to the classical Kaczmarz method. Perhaps the oldest known convergence result can be expressed as $\norm{\vx_k - \vx^*}^2 \leq \big(1 - \mathrm{det}(\hat{\mA}\hat{\mA}^\top)\big)^k \norm{\vx_0 - \vx^*}^2$ \cite{smith1977practical}. This bound can be arbitrarily loose, as discussed in \cite{sun2021worst}.  Another quantitative bound on the convergence of Kaczmarz method is implied by the results from \cite{Mandel84}, leading to a convergence bound of the form $\norm{\vx_k - \vx^*}^2 \leq \big(1 - \frac{\sigma_r^2(\hat{\mA})}{m^2}\big)^k \norm{\vx_0 - \vx^*}^2$; see also \cite{Dai_2015}. This is a coarse estimate of convergence, oblivious to structural properties of $\mA$ and explicitly scaling with the problem dimension. A slightly tighter estimate also follows from \cite{Mandel84}, replacing the contraction factor $1 - \frac{\sigma_r^2(\hat{\mA})}{m^2}$ by $1 - \frac{\mu^2}{m},$ where $\mu = \inf_{\vu \in \range(\mA^\top), \norm{\vu} = 1}\norm{\hat{\mA}\vu}_\infty;$ see \cite{chen2018kaczmarz}. 

 The first convergence bound for the classical Kaczmarz method that involves singular values and the rank of $\mA$ is due to \cite{Oswald_2015,oswald2017random}, giving convergence bound of the form  $\norm{\vx_k - \vx^*}^2 \leq \Big(1-\frac{\sigma_r^2(\hat{\mA})}{\big(1+C(1+\ln(r))\sigma_1^2(\hat{\mA})\big)^2}\Big)^k \norm{\vx_0 - \vx^*}^2$, where $C > 0$ is a universal constant (unspecified in \cite{Oswald_2015,oswald2017random}) and $r = \rank(\mA)$. This bound follows from a more general bound for the relaxed method. However, as noted in the same work, this bound provides only a worst-case guarantee because it does not depend on row ordering and therefore does not capture the ordering-specific structure of $\mA$. A slightly tighter bound with better, fully specified constants can be recovered from our results (see \Cref{lem:L-I-bnd} and \Cref{cor:row-correlation-bound}; note that our bounds also capture row correlations, which are order-specific, and are new). 

 Another related convergence bound $\norm{\vx_k - \vx^*}^2 \leq \big(1 - \max\big\{\frac{\sigma_r^2(\hat{\mA})}{m\sigma_1^2(\hat{\mA})}, \, \frac{\sigma_r^2(\hat{\mA})}{\sigma_1^4(\hat{\mA})(2 + \log(m)/\pi)^2}\big\}\big)\norm{\vx_0 - \vx^*}^2$ is due to \cite{sun2021worst}. The bounds provided in our work are strictly tighter. This is clear for the second term in their bound, due to $\log_2(r)$ in our bounds in place of their $\log(m)$. For the first term in the maximum, we claim that $\frac{\sigma_r^2(\hat{\mA})}{\norm{\mL- \mI_m}_\op^2 + \sigma_r^2(\hat{\mA})}$ in our bound gives a larger estimate, and thus a tighter bound. The reason is that 
 \[\norm{\mL - \mI_m}_\op^2 \leq \sum_{i<j}\innp{\hat{\va}_i, \hat{\va}_j}^2 = \frac{1}{2}\big(\sum_{i=1}^r \sigma_i^4(\hat{\mA}) - m\big) \leq \frac{m}{2}(\sigma_1^2(\hat{\mA})-1),\]
 where we have used $\sum_{i=1}^r \sigma_i^4(\hat{\mA}) \leq \sigma_1^2(\hat{\mA})\sum_{i=1}^r \sigma_i^2(\hat{\mA}) = m \sigma_1^2(\hat{\mA})$ (as the matrix $\hat{\mA}$ is row-normalized). For $r \geq 2$, $\norm{\mL - \mI_m}_\op^2 + \sigma_r^2(\hat{\mA}) \leq \frac{m}{2}(\sigma_1^2(\hat{\mA}) - 1) + \sigma_r^2(\hat{\mA}) \leq \frac{m}{2}\sigma_1^2(\hat{\mA}),$ as $\sigma_r^2(\hat{\mA}) \leq m/r,$ so our bound is tighter by a factor at least two. (For $r=1$, our bound correctly predicts convergence in a single cycle, thus is trivially tighter.) Moreover, $\norm{\mL- \mI_m}_\op^2$ preserves row order dependence and can generally be much smaller than the provided upper bound. 

 While most of the focus in prior work has been on quantifying linear rates of convergence, there are far fewer results concerning regimes where sublinear rates may provide tighter estimates across a fixed number of cycles \cite{evron2022catastrophic,swartworth2023nearly,reich2023polynomial}. Most relevant to our work is \cite{evron2022catastrophic}. While their work considers more general settings motivated by continual learning, when specialized to classical Kaczmarz method, their convergence bound is, assuming $m \leq n$ and $k \geq m,$ $\norm{\hat{\mA} (\vx_k - \vx^*)}^2 \leq \min\big\{\frac{m^{5/2}}{\sqrt{k}},\, \frac{m^2(r-1)}{2 k}\big\}\norm{\vx_0 - \vx^*}^2,$ so translated to the original matrix residual, it becomes $\norm{{\mA} \vx_k - \vb}^2 \leq \max_{i\in[m]}\norm{\va_i}^2\min\big\{\frac{m^{5/2}}{\sqrt{k}},\, \frac{m^2(r-1)}{2 k}\big\}\norm{\vx_0 - \vx^*}^2.$ This is a much coarser bound than $\frac{ (r -1)\max_{i \in [m]}\norm{\va_i}^2\norm{\mL - \mI_m}_\op^2 \norm{\vx_0 - \vx^*}^2}{k}$ from \Cref{thm:matrix-power-bound}, as $\norm{\mL - \mI_m}_\op^2 \leq \frac{m(m-1)}{2}$. 

In addition to the discussed convergence results for Kaczmarz method, there are other lines of work that relate to ours on a technical level. As mentioned in \Cref{sec:prelims}, the inspiration for tracking the projections of $\vx_k - \vx^*$ onto the space spanned by the right singular vectors of $\mA$, $\ve_k = \mV^\top(\vx_k - \vx^*),$ came from the recent results concerning convergence of randomized Kaczmarz method \cite{steinerberger2021randomized}. However, this is where the overlap with the work on randomized Kaczmarz ends. Our intuition for bringing the evolution of the projected vectors $\ve_k$ into a form involving a triangular matrix came from the recent analyses of cyclic \cite{song2023cyclic} and incremental gradient \cite{cai2024tighter} methods where such matrices arise from relating within-cycle progress to the full-cycle quantities (see also \cite{cai2026near} for another example where tracking progress in the aggregate, across the full cycle of updates, turns out to be crucial). Finally, we note that lower triangular matrices also arise in prior analyses of Kaczmarz method, notably in \cite{Oswald_2015}. Here we point out that, since the obtained relationships are algebraically equivalent to the Kaczmarz update \eqref{eq:Kaczmarz-update-full}, it is not surprising that related quantities would arise before any relaxations are made. In a similar fashion, \cite{xu2002method} provides a tight variational characterization of convergence of POCS (which includes Kaczmarz method as a special case), in terms of quantities involving projection operators. However, we note that, crucially, the insights leading to tight convergence results that are \emph{interpretable} come from identifying the ``right'' quantities to track, so in that respect algebraic equivalence is irrelevant. 
 
 Finally, a recent work \cite{hansen2026spectral} studied the cyclic sweep operator of the form $\mG=\mI_n-\mA^\top\mL^{-1}\mA$ for consistent, possibly rank-deficient inverse problems, where $\mL$ is a lower-triangular matrix similar in structure to $\mL$ defined in \eqref{eq:L-and-M-defs}, and $\mG$ is related to our $\mM;$ namely,  $\mM = \mV^\top\mG \mV.$ 
 Their result uses $\mG$ to reason about fast initial convergence and row-order dependence; however, it does not provide any explicit, nonasymptotic convergence bounds.

\section{Main results}

In this section, we state and prove all our results concerning tight convergence characterizations for the classical, deterministic cyclic Kaczmarz update stated in \eqref{eq:Kaczmarz-update-full}. We begin the discussion by proving a core technical lemma (\Cref{lem:dynamics-of-ek}), which leads to exactly tight convergence bounds that were stated in \Cref{thm:matrix-power-bound}. We then provide bounds on quantities involving auxiliary matrices $\mL$ and $\mM$ (defined in \eqref{eq:L-and-M-defs}), based on which the tight bounds are expressed, obtaining more interpretable bounds in terms of the matrix spectrum, row norms and correlations, and rank (\Cref{sec:interpretable-bounds}). Finally, in \Cref{sec:tightness}, we provide simple illustrative examples that demonstrate tightness of the obtained (relaxed, interpretable) convergence bounds.  

\begin{lemma}[Main Lemma]\label{lem:dynamics-of-ek}
    Consider a consistent linear system \eqref{eq:main-prob}, where $\rank(\mA) = r.$ 
    For $\ve_k :=\mV^\top (\vx_k-\vx^*)$, where $\vx^*$ is the solution closest to $\vx_0$ (see \eqref{eq:x^*-specification}), we have 
    \[
    \ve_{k+1}=(\mI_r - (\hat{\mA}\mV)^\top \mL^{-1}\hat{\mA}\mV)\ve_k = \mM\ve_k,
    \]
    where $\mL$ and $\mM$ were defined in \eqref{eq:L-and-M-defs}.
Furthermore, $$\|\ve_{k+1}\|^2 = \|\ve_k\|^2 - \| \mL^{-1} \hat{\mA} \mV \ve_k\|^2 \le \|\ve_k \|^2,$$
and, as a result, $\| \mM\|_\op < 1$. 
\end{lemma}
\begin{proof}
    The Kaczmarz update stated in the normalized form \eqref{eq:Kacmarz-update-normalized} gives 
    \(
    \vx_{k+1}-\vx^* = \vx_k - \vx^* - \sum_{i=1}^m \innp{\hat{\va}_i, \vx_{k,i-1}-\vx^*} \hat{\va}_i.
    \) 
    Taking the inner product of both sides with $\vv_\ell$ then leads to
    \begin{equation}\notag
        \ve_{k+1}^{(\ell)} = \ve_{k}^{(\ell)} - \sum_{i=1}^m \innp{\hat{\va}_i, \vx_{k,i-1}-\vx^*} \innp{\hat{\va}_i, \vv_\ell}.
    \end{equation}
    To accurately track the evolution of $\ve_k$ and within-cycle cancellations, inspired by \cite{cai2024tighter}, the idea is to use an auxiliary lower triangular matrix $\mL$ defined in the lemma statement.  
    Observe that $\innp{\hat{\va}_i, \vv_\ell} = (\hat{\mA} \vv_\ell)_i$. Define $\vs \in \R^m$ by $\vs^{(i)} = \innp{\hat{\va}_i, \vx_{k,i-1}-\vx^*}$, and let $\vd_i = \vx_{k,i} - \vx^*$. Then $\vs^{(i)} = \innp{\hat{\va}_i, \vd_{i-1}}$ and, by \eqref{eq:Kaczmarz-update-within-cycle} and the definition $\hat{\va},$ 
    \begin{equation}\label{eq:recursion-on-d}
        \vd_i = \vd_{i-1} - \innp{ \hat{\va}_i, \vd_{i-1}}\hat{\va}_i = \vd_{i-1} - \vs^{(i)} \hat{\va}_i.
    \end{equation}
    We claim that, for any $i \in [m]$, 
    \begin{equation}\label{eq:cumulative-s}
        \sum_{j=1}^{i-1}\innp{\hat{\va}_i, \hat{\va}_j}\vs^{(j)} + \vs^{(i)} = \innp{\hat{\va}_i, \vd_0} = \innp{\hat{\va}_i, \vx_k - \vx^*}.
    \end{equation} 
    To prove this claim, first note that, by definition, $\vs^{(1)} = \innp{\hat{\va}_1, \vd_0}$.
    For $i\ge2$, by unrolling the recursion in \eqref{eq:recursion-on-d}, 
    \begin{align*}
        \vs^{(i)} = \innp{\hat{\va}_i, \vd_{i-1}} 
        = \innp{\hat{\va}_i, \vd_0 - \sum_{j=1}^{i-1}\vs^{(j)} \hat{\va}_j} 
        = \innp{\hat{\va}_i, \vd_0} - \sum_{j=1}^{i-1} s^{(j)}\innp{\hat{\va}_i, \hat{\va}_j}.
    \end{align*} 
    The claimed identity \eqref{eq:cumulative-s} follows by moving the summation term to the left-hand side.

    Stacking the identities \eqref{eq:cumulative-s} for all $i \in [m]$ leads to 
    \(
    \mL \vs = \hat{\mA} \vd_0 = \hat{\mA}(\vx_k - \vx^*). 
    \) 
    Since $\mL$ is lower triangular with all diagonal entries equal to $1$, it is invertible. Thus,
    \begin{equation}\label{eq:s-identity-1}
        \vs = \mL^{-1} \hat{\mA}(\vx_k-\vx^*).
    \end{equation}
    Furthermore, as noted in \Cref{sec:prelims}, the right singular vectors $\vv_1, \cdots, \vv_r$ span $\range(\mA^\top)$, and by the choice of $\vx^*,$ $\vx_k-\vx^* \in \range(\mA^\top)$ (see \eqref{eq:needed-gen-V}). Therefore, $\vx_k-\vx^* =
    \sum_{j=1}^r \ve_{k}^{(j)} \vv_j.$ 
    Substituting this representation into \eqref{eq:s-identity-1} gives
    \begin{equation}\label{eq:s-indentity-2}
        \vs = \mL^{-1}\hat{\mA} \sum_{j=1}^r \ve_k^{(j)}\vv_j = \sum_{j=1}^r \ve_k^{(j)}(\mL^{-1}\hat{\mA}\vv_j).
    \end{equation}

    Consequently,
    \begin{align*}
        \ve_{k+1}^{(\ell)} &= \ve_{k}^{(\ell)} - \sum_{i=1}^m \innp{\hat{\va}_i, \vx_{k,i-1}-\vx^*} \innp{\hat{\va}_i, \vv_\ell} \\
        &\stackrel{(i)}{=} \ve_k^{(\ell)}- \innp{\vs, \hat{\mA}\vv_\ell} \\
        &\stackrel{(ii)}{=} \ve_k^{(\ell)} - \sum_{j=1}^r \innp{\mL^{-1}(\hat{\mA}\vv_j), \hat{\mA}\vv_\ell}\ve_k^{(j)},
    \end{align*}
    where $(i)$ is by the definition of $\vs$ and $(ii)$ follows by \eqref{eq:s-indentity-2}. 
    Writing these relations in matrix form now gives
    $\ve_{k+1}=(\mI_r - (\hat{\mA}\mV)^\top \mL^{-1}\hat{\mA}\mV)\ve_k,$ completing the proof of the first lemma claim. 

    The second lemma claim follows by telescoping. Indeed, since 
    $\vd_i = \vd_{i-1} - \vs^{(i)} \hat{\va}_i$, we have 
    \begin{align*}
        \| \vd_i\|^2 &= \| \vd_{i-1}\|^2 - 2 \vs^{(i)} \innp{\vd_{i-1}, \hat{\va}_i} + (\vs^{(i)})^2 \|\hat{\va}_i \|^2 \\
        &= \| \vd_{i-1}\|^2 -2 (\vs^{(i)})^2 + (\vs^{(i)})^2 \\
        &= \| \vd_{i-1}\|^2-(\vs^{(i)})^2.
    \end{align*}
    Summing over $i\in [m]$ gives
    \begin{align*}
        \|\ve_{k+1}\|^2 - \|\ve_k \|^2 &= \| \mV^\top \vd_m \|^2 - \|\mV^\top \vd_0\|^2 \\ 
        &= \sum_{i=1}^m\big( \|\vd_i \|^2 - \| \vd_{i-1}\|^2\big) \\
        &= -\sum_{i=1}^m (\vs^{(i)})^2 
        = -\|\vs\|^2 \\
        &= -\|\mL^{-1} \hat{\mA}(\vx_k - \vx^*) \|^2 = -\|\mL^{-1}\hat{\mA}\mV \ve_k \|^2,
    \end{align*}
    where the last equality is by again using $\vx_k - \vx^* = \sum_{\ell=1}^r\innp{\vx_k - \vx^*, \vv_\ell}\vv_\ell = \mV\ve_k.$ 
    We finally conclude that since $\ve_{1} = \mM\ve_0$ and $\norm{\ve_{1}}^2 \leq \norm{\ve_0}^2 -\|\mL^{-1} \hat{\mA}\mV\ve_0 \|^2$, which is $< \norm{\ve_0}^2$ whenever $\ve_0 \neq \vzero$ as $\hat{\mA}\mV$ has full column rank and $\mL$ is invertible, it must be $\norm{\mM}_{\op} < 1,$ as $\vx_0$ (and thus $\ve_0$) can be arbitrary. 
\end{proof}

Observe that \Cref{lem:dynamics-of-ek} gives the exact evolution of the solution residual $\vx_k - \vx^*$ in the basis induced by right singular vectors $\mV,$ via relation $\ve_{k + 1} = \mM \ve_k$, $\mM := \mI_r - (\hat{\mA}\mV)^\top \mL^{-1}\hat{\mA}\mV,$ for all $k \geq 0.$ Importantly, the matrix $\mM$ is fully specified by the input data, as it only depends on quantities that are fixed and fully determined by the input matrix $\mA.$ In that sense, \Cref{lem:dynamics-of-ek} tightly characterizes Kaczmarz updates. 

A useful identity relating $\mL$ and $\hat{\mA}$ is stated in the following claim. It is used in deriving quantitative convergence results in \Cref{thm:matrix-power-bound} and \Cref{cor:row-correlation-bound}.

\begin{claim}\label{claim:aux-L-hatA}
    Given a matrix $\mA$ and its row-normalized form $\hat{\mA}$, let $\mL$ be defined as in \Cref{lem:dynamics-of-ek}. Then $\hat{\mA}\hat{\mA}^\top = \mL + \mL^\top - \mI_m$, and, as a consequence, $\mL \mL^\top = (\mL -\mI_m)(\mL - \mI_m)^\top + \hat{\mA}\hat{\mA}^\top.$
\end{claim}
\begin{proof}
    Consider $\hat{\mA}\hat{\mA}^\top$. Its $(i,j)$ entry is $(\hat{\mA}\hat{\mA}^\top)^{(i,j)} = \innp{\hat{\va}_i, \hat{\va}_j}$, while $$L^{(i,j)}= \begin{cases}
        \innp{\widehat{\va}_i,\widehat{\va}_j}, & j < i, \\
        1, & j = i, \\
        0, & j > i.
    \end{cases}$$
    Thus, $\hat{\mA}\hat{\mA}^\top = \mL + \mL^\top - \mI_m$. It remains to derive 
    \begin{align*}
         \mL \mL^\top &= (\mL -\mI_m + \mI_m)(\mL -\mI_m + \mI_m)^\top \\
        &=(\mL -\mI_m)(\mL - \mI_m)^\top + \hat{\mA}\hat{\mA}^\top,
    \end{align*}
    and the proof is complete.
\end{proof}

We are now ready to prove \Cref{thm:matrix-power-bound}. 
\begin{proof}[Proof of \Cref{thm:matrix-power-bound}]
    The first three equalities (in \eqref{eq:tight-contraction}, \eqref{eq:tight-telescoping}) are a direct consequence of \Cref{lem:dynamics-of-ek}. In particular, since any $\ve_0 \in \sR^r$ can be realized via $\vx_0 = \vx^* + \mV\ve_0,$ \Cref{lem:dynamics-of-ek} can be invoked with $\ve_0$ chosen as the top unit right singular vector of $\mM^k$ or $\mSigma \mM^k.$ Equality \eqref{eq:tight-telescoping} follows from $\|\ve_{k+1}\|^2 = \|\ve_k\|^2 - \| \mL^{-1} \hat{\mA} \mV \ve_k\|^2 \le \|\ve_k \|^2$ and \eqref{eq:needed-gen-V}, which leads to 
    \begin{equation}\label{eq:telescoping-for-distance-to-opt}
    \norm{\vx_{k+1} - \vx^*}^2 = \norm{\vx_k - \vx^*}^2 - \norm{\mL^{-1}\hat{\mA}(\vx_{k} - \vx^*)}^2,
    \end{equation}
    and so all that is needed is telescoping and a simple rearrangement.

Further, by the same lemma, $\ve_k = \mM^k \ve_0$, where $\ve_0 = \mV^\top(\vx_0-\vx^*)$. 
Since $\vb=\mA\vx^*$, $\mA=\mU\mSigma\mV^\top$, $\mU$ has orthonormal columns, and \eqref{eq:needed-gen-V} holds, we have
    \begin{align*}
        \| \mA \vx_k - \vb\|^2
        &= \|\mA(\vx_k - \vx^*)\|^2 
         = \|\mSigma\ve_k\|^2 
        = \|\mSigma \mM^k \ve_0\|^2 \\
        &\leq \norm{\mSigma \mM}_\op^2 \|\mM^{k-1}\|_\op^2 \|\vx_0-\vx^*\|^2.
    \end{align*}
    This gives the first bound inside the minimum. 

    To obtain the remaining two bounds on the residual, we recall, from \Cref{lem:dynamics-of-ek}, that $\ve_{k+1} - \ve_k = -(\hat{\mA}\mV)^\top \mL^{-1}\hat{\mA}\mV \ve_k$ and $\ve_k = \mV^\top(\vx_k - \vx^*).$ Thus,
    \[
    \mV^\top (\vx_{k+1} - \vx_k) = -(\hat{\mA}\mV)^\top \mL^{-1}\hat{\mA}\mV \mV^\top(\vx_k - \vx^*).
    \]
    Multiplying both sides by $\hat{\mA}\mV$ and using that $\hat{\mA}\mV\mV^\top = \hat{\mA}$ and $(\hat{\mA}\mV)(\hat{\mA}\mV)^\top = \hat{\mA}\hat{\mA}^\top,$ we get
    \(
        \hat{\mA}(\vx_{k+1} - \vx_k) = -\hat{\mA}\hat{\mA}^\top \mL^{-1}\hat{\mA}(\vx_k - \vx^*),
    \) 
    and so 
    \[
        \hat{\mA}(\vx_{k+1} - \vx^*) = \big(\mI_m -\hat{\mA}\hat{\mA}^\top \mL^{-1}\big)\hat{\mA}(\vx_k - \vx^*)
    \]
    Now, from \Cref{claim:aux-L-hatA}, $\hat{\mA}\hat{\mA}^\top = \mL + \mL^\top - \mI_m,$ leading to
    \begin{align}
        \hat{\mA}(\vx_{k+1} - \vx^*) &= \big(\mI_m - (\mL + \mL^\top - \mI_m)\mL^{-1}\big)\hat{\mA}(\vx_k - \vx^*)\notag\\
        &= -(\mL - \mI_m)^\top \mL^{-1}\hat{\mA}(\vx_k - \vx^*). \label{eq:residual-identity} 
    \end{align}
    As a consequence,
    \begin{equation}\label{eq:move-rescaled-residual}
        \norm{\hat{\mA}(\vx_{k+1} - \vx^*)}^2 \leq \norm{\mL - \mI_m}_\op^2 \norm{\mL^{-1}\hat{\mA}(\vx_k - \vx^*)}^2. 
    \end{equation}
    Combining \eqref{eq:move-rescaled-residual} with \eqref{eq:tight-telescoping} and simplifying, we have
    \begin{equation}\notag
        \frac{1}{k+1}\sum_{i=0}^{k} \norm{\hat{\mA}(\vx_{i+1} - \vx^*)}^2 \leq \frac{\norm{\mL - \mI_m}_\op^2\norm{\vx_0 - \vx^*}^2}{k+1},
    \end{equation}
    so all that remains to prove the last bound is to observe that $\mA = \diag(\norm{\va_1}, \dots, \norm{\va_m})\hat{\mA},$ and so 
    \begin{equation}\label{eq:rescaling}
        \norm{\mA(\vx_{i+1} - \vx^*)} \leq \max_{i \in [m]}\norm{\va_i}\norm{\hat{\mA}(\vx_{i+1} - \vx^*)}.
    \end{equation} 

    For the remaining sublinear rate bound on the last iterate, observe first that combining \eqref{eq:telescoping-for-distance-to-opt}  and \eqref{eq:move-rescaled-residual} leads to $\norm{\hat{\mA}(\vx_{k+1} - \vx^*)}^2 \leq \norm{\mL - \mI_m}_\op^2 (\norm{\vx_k - \vx^*}^2 - \norm{\vx_{k+1} - \vx^*}^2),$ and so using \eqref{eq:rescaling}, we have
    \begin{equation}\notag
        \norm{\mA(\vx_{k+1} - \vx^*)}^2 \leq \max_{i \in [m]}\norm{\va_i}^2\norm{\mL - \mI_m}_\op^2 (\norm{\vx_k - \vx^*}^2 - \norm{\vx_{k+1} - \vx^*}^2). 
    \end{equation}
    By \eqref{eq:needed-gen-V} and using $\ve_k = \mM^k \ve_0$, the last inequality is equivalent to
    \begin{equation}\label{eq:telescoping-in-M}
        \norm{\mSigma \mM^{k+1}\ve_0}^2 \leq \max_{i \in [m]}\norm{\va_i}^2\norm{\mL - \mI_m}_\op^2 (\norm{\mM^k\ve_0}^2 - \norm{\mM^{k+1}\ve_0}^2).
    \end{equation}
    Note that \eqref{eq:telescoping-in-M} was derived for arbitrary $k \geq 0$ and arbitrary initial vector $\vx_0$, so it applies to arbitrary $\ve_0 \in \sR^r$ ($\mSigma$ and $\mM$ only depend on the matrix $\mA;$ they are independent of $\vx_0$). Further, we argue that $\mM \mV^\top\hat{{\va}}_1 = \vzero,$ so $\rank(\mM) \leq r -1.$ To see this, observe that $\hat{\mA}\hat{\va}_1 = (1, \innp{\hat{\va}_1, \hat{\va_2}}, \dots, \innp{\hat{\va_1}, \hat{\va}_m})^\top = \mL \vi_1$, where $\vi_1 := (1, 0, \dots, 0)^\top$. Thus,
    \begin{align*}
        \mM \mV^\top\hat{\va}_1 &= \mV^\top\hat{\va}_1 - (\hat{\mA}\mV)^\top \mL^{-1}\hat{\mA}\mV \mV^\top\hat{\va}_1\\
        &= \mV^\top\hat{\va}_1 - (\hat{\mA}\mV)^\top \mL^{-1}\hat{\mA}\hat{\va}_1\\
        &= \mV^\top\hat{\va}_1 - (\hat{\mA}\mV)^\top \mL^{-1}\mL\vi_1 = \vzero.
    \end{align*}
    Thus, for $\mP_1 = \mI_r - (\mV^\top\hat{\va}_1)(\mV^\top\hat{\va}_1)^\top,$ we have $\mM = \mM\mP_1.$ Note here that $\norm{\mV^\top \hat{\va}_1} = 1.$
    
    Applying \eqref{eq:telescoping-in-M} to each column of $\mP_1$ and summing, we get
    \begin{equation}\label{eq:telescoping-in-M-F}
        \norm{\mSigma \mM^{k+1}\mP_1}_F^2 \leq \max_{i \in [m]}\norm{\va_i}^2\norm{\mL - \mI_m}_\op^2 (\norm{\mM^k\mP_1}_F^2 - \norm{\mM^{k+1}\mP_1}_F^2).
    \end{equation}
    Since \eqref{eq:telescoping-in-M-F} applies to all $k \geq 0$, telescoping it and then dropping the non-positive term $- \norm{\mM^{k+1}\mP_1}_F^2$, we arrive at
    \begin{equation}\label{eq:M-space-after-telescoping}
        \sum_{i=1}^{k+1} \norm{\mSigma \mM^{i}\mP_1}_F^2 \leq \max_{i \in [m]}\norm{\va_i}^2\norm{\mL - \mI_m}_\op^2 \norm{\mM^0\mP_1}_F^2 \leq (r-1) \max_{i \in [m]}\norm{\va_i}^2 \norm{\mL - \mI_m}_\op^2, 
    \end{equation}
    where the last inequality follows from $\norm{\mP_1}_F^2 = \rank(\mP_1) = r-1.$ 
    Additionally, for $i \geq 1,$ 
    \[
    \norm{\mSigma \mM^{i+1}\mP_1}_F = \norm{\mSigma \mM^i \mM}_F \leq \norm{\mSigma \mM^{i}}_F\norm{\mM}_\op \leq \norm{\mSigma \mM^{i}}_F = \norm{\mSigma \mM^{i}\mP_1}_F,
    \] 
    and so \eqref{eq:M-space-after-telescoping} implies
    \begin{equation}\notag
        \norm{\mSigma \mM^{k+1}\mP_1}_F^2 \leq \frac{(r-1) \max_{i \in [m]}\norm{\va_i}^2 \norm{\mL - \mI_m}_\op^2}{k+1}. 
    \end{equation}
    To complete the proof, it remains to recall that $\mM^{k+1}\mP_1 = \mM^{k+1}$ and use \eqref{eq:tight-contraction} to deduce $\norm{\mSigma \mM^{k+1}\mP_1}_F^2 = \norm{\mSigma \mM^{k+1}}_F^2 \geq \norm{\mSigma \mM^{k+1}}_\op^2 \geq \frac{\norm{\mA(\vx_{k+1} - \vx^*)}^2}{\norm{\vx_0 - \vx^*}^2}$, where we have taken, without loss of generality, that $\vx_0 \neq \vx^*$ (otherwise the claimed bound holds trivially).  
\end{proof}

\subsection{More Interpretable Convergence Bounds}\label{sec:interpretable-bounds}

Although our main convergence results, stated in terms of $\norm{\mM^k}_\op$ are tight, they do not provide an explicit, interpretable, linear convergence. To obtain additional quantitative bounds on the convergence rate determined by $\norm{\mM^k}_\op$, we provide a relaxation that translates into an explicit convergence rate of the form $q^k,$ $q \in [0, 1)$, by bounding the operator norm of $\mM$. We discuss tightness of this relaxation, as well as the relaxed convergence bounds from  \Cref{thm:matrix-power-bound}, in the next subsection.  

\begin{lemma}\label{lem:spectrum-of-SigmaM}
    Given a matrix $\mA = \mU \mSigma \mV^\top,$ let $\mM = \mI_r - (\hat{\mA}\mV)^\top \mL^{-1}\hat{\mA}\mV$ and let $\mP_m := \mI_n - \hat{\va}_m \hat{\va}_m^\top$. Then:
    \begin{equation}\notag
        \norm{\mSigma \mM}_\op^2 \leq \min\Big\{\norm{\mA\mP_m}_\op^2\norm{\mM}_\op^2, \,  \frac{\max_{i \in [m]}\norm{\va_i}^2\norm{\mL - \mI_m}_\op^2\norm{\mA\mP_m}_\op^2}{\norm{\mA\mP_m}_\op^2 + \max_{i \in [m]}\norm{\va_i}^2\norm{\mL - \mI_m}_\op^2}\Big\}, 
    \end{equation}
    where we additionally have $\norm{\mA \mP_m}_\op^2 \leq \norm{\mA\mP_m}_F^2 = \sum_{i=1}^m \norm{\va_i}^2\big(1 - \innp{\hat{\va}_i, \hat{\va}_m}^2\big)$ and $\norm{\mA \mP_m}_\op \leq \sigma_1(\mA)$.
\end{lemma}
\begin{proof}
By \Cref{lem:dynamics-of-ek}, $\ve_1 = \mM \ve_0$, where $\ve_0 = \mV^\top(\vx_0-\vx^*)$. Kaczmarz update \eqref{eq:Kacmarz-update-normalized} ensures that the last equality in the system is solved exactly at the end of the cycle, so $\innp{\hat{\va}_m, \vx_1 - \vx^*} = 0.$ Let $\mP_m := \mI_n - \hat{\va}_m\hat{\va}_m^\top.$ Then $\mV \ve_1 = \mP_m \mV \ve_1.$ Equivalently, $\mV\mM \ve_0 = \mP_m \mV \mM \ve_0.$ Multiplying both sides from the left by $\mA$ and taking the operator norm, we have
\begin{align*}
    \norm{\mA \mV \mM \ve_0} &= \norm{\mA \mP_m \mV \mM \ve_0}\\
    &\leq \norm{\mA \mP_m}_\op \norm{\mV \mM \ve_0}\\
    &\leq \norm{\mA \mP_m}_\op \norm{\mM}_\op \norm{\ve_0}.
\end{align*}
Since the above inequality holds for arbitrary $\ve_0 \in \sR^r,$ we can, in particular, take the supremum over $\|\ve_0\| = 1,$ so we get $\norm{\mA \mV \mM}_\op \leq \norm{\mA \mP_m}_\op \norm{\mM}_\op.$ Finally, since $\mA = \mU \mSigma \mV^\top,$ we have $\mA \mV = \mU \mSigma$, and using orthonormality of the columns of $\mU,$ we get  $\norm{\mA \mV \mM}_\op = \norm{\mSigma \mM}_\op,$ leading to  $\norm{\mSigma \mM}_\op \leq \norm{\mA \mP_m}_\op \norm{\mM}_\op.$

For the remaining bound inside the minimum, suppose that $\norm{\mA \mP_m}_\op \neq 0,$ for otherwise the lemma claim holds trivially. Observe first that, as already noted by $\innp{\hat{\va}_m, \vx_1 - \vx^*} = 0,$ we have $\mA(\vx_1 - \vx^*) = \mA \mP_m (\vx_1 - \vx^*)$ and so 
\begin{equation}\label{eq:bnd-r1}
\norm{\mA(\vx_1 - \vx^*)} \leq \norm{\mA \mP_m}_\op \norm{\vx_1 - \vx^*}. 
\end{equation}
From \eqref{eq:residual-identity}, $\hat{\mA}(\vx_{k+1} - \vx^*) = -(\mL - \mI_m)^\top \mL^{-1}\hat{\mA}(\vx_k - \vx^*)$; thus, we have 
\[
\norm{\mA (\vx_1 - \vx^*)} \leq \max_{i \in [m]}\norm{\va_i}\norm{\hat{\mA}(\vx_1 - \vx^*)} \leq \max_{i \in [m]}\norm{\va_i}\norm{\mL - \mI_m}_\op \norm{\mL^{-1}\hat{\mA}(\vx_0 - \vx^*)}.
\]
But $\norm{\mL^{-1}\hat{\mA}(\vx_0 - \vx^*)}^2 = \norm{\mL^{-1}\hat{\mA}\mV \ve_0}^2 = \norm{\ve_0}^2 - \norm{\ve_1}^2,$ by \Cref{lem:dynamics-of-ek}, which is further equal to $\norm{\vx_0 - \vx^*}^2 - \norm{\vx_1 - \vx^*}^2.$ We can thus conclude that
\begin{align*}
\norm{\mA (\vx_1 - \vx^*)}^2 &\leq \max_{i \in [m]}\norm{\va_i}^2\norm{\mL - \mI_m}_\op^2 \big(\norm{\vx_0 - \vx^*}^2 - \norm{\vx_1 - \vx^*}^2\big)\\
&\leq \max_{i \in [m]}\norm{\va_i}^2\norm{\mL - \mI_m}_\op^2 \big(\norm{\vx_0 - \vx^*}^2 - \norm{\mA(\vx_1 - \vx^*)}^2/\norm{\mA \mP_m}_\op^2\big), 
\end{align*}
where the last inequality is by \eqref{eq:bnd-r1}. A rearrangement of the last inequality now gives
\begin{equation}\notag
    \norm{\mA (\vx_1 - \vx^*)}^2 \leq \frac{\max_{i \in [m]}\norm{\va_i}^2\norm{\mL - \mI_m}_\op^2\norm{\mA\mP_m}_\op^2}{\norm{\mA\mP_m}_\op^2 + \max_{i \in [m]}\norm{\va_i}^2\norm{\mL - \mI_m}_\op^2}\norm{\vx_0 - \vx^*}^2.
\end{equation}
To complete the proof, it remains to use again that $\mA (\vx_1 - \vx^*) = \mA \mV \ve_1 = \mA \mV \mM \ve_0 = \mU \mSigma \mM \ve_0$ and set $\ve_0$ to the unit top right singular vector of $\mSigma \mM.$ 

The identity $\norm{\mA\mP_m}_F^2 = \sum_{i=1}^m \norm{\va_i}^2\big(1 - \innp{\hat{\va}_i, \hat{\va}_m}^2\big)$ follows by a direct calculation. 
\end{proof}

One way to relax $\|\mM^k\|_\op$ is by using submultiplicativity and bounding the spectral norm of $\mM$, as shown in the following lemma.
\begin{lemma}\label{lem:spectrum-of-M}
    Let $\mM = \mI_r - (\hat{\mA}\mV)^\top \mL^{-1}\hat{\mA}\mV$. Then %
    \[\
    \|\mM\|_\op^2 \le\frac{\| \mL - \mI_m\|_\op^2}{\| \mL - \mI_m\|_\op^2+\sigma_r^2(\hat{\mA})} \leq \frac{\max_{i \in [m]}\norm{\va_i}^2\| \mL - \mI_m\|_\op^2}{\max_{i \in [m]}\norm{\va_i}^2\| \mL - \mI_m\|_\op^2+\sigma_r^2(\mA)}.
    \]
\end{lemma}
\begin{proof}
Because $\mV$ spans the row space of $\mA,$ we have that $\hat{\mA}\mV$ has the full column rank $r,$ with minimum singular value $\sigma_r(\hat{\mA})\geq \frac{\sigma_r(\mA)}{\max_{i \in [m]}\norm{{\va}_i}}.$ Further, from \eqref{eq:residual-identity} (using \eqref{eq:needed-gen-V}), $\hat{\mA}\mV \mM\ve_0  = -(\mL - \mI_m)^\top\mL^{-1}\hat{\mA}\mV\ve_0.$ Thus, we get
\begin{equation}\label{eq:Me-ub}
    \norm{\hat{\mA}\mV \mM\ve_0}^2 \leq \norm{\mL - \mI_m}_\op^2 \norm{\mL^{-1}\hat{\mA}\mV\ve_0}^2 \leq \norm{\mL - \mI_m}_\op^2\big(\norm{\ve_0}^2 - \norm{\mM\ve_0}^2\big).
\end{equation}
Further, since $\norm{\hat{\mA}\mV \mM \ve_0} \geq \sigma_r(\hat{\mA})\norm{\mM \ve_0},$ combining with \eqref{eq:Me-ub}, we get \[(\sigma_r^2(\hat{\mA}) + \norm{\mL - \mI_m}_\op^2)\norm{\mM \ve_0}^2 \leq \norm{\mL - \mI_m}_\op^2 \norm{\ve_0}^2.\] Since $\ve_0$ can be chosen arbitrarily (as $\vx_0$ can be chosen arbitrarily), we can set it to the top unit right singular vector of $\mM$ to get $\|\mM\|_\op^2 \le\frac{\| \mL - \mI_m\|_\op^2}{\| \mL - \mI_m\|_\op^2+\sigma_r^2(\hat{\mA})}.$ The remaining inequality follows from $\sigma_r(\hat{\mA})\geq \frac{\sigma_r(\mA)}{\max_{i \in [m]}\norm{{\va}_i}}.$
\end{proof}

The preceding spectral-norm estimate gives an explicit consequence of \Cref{thm:matrix-power-bound}, stated in the \Cref{cor:row-correlation-bound}. This estimate is particularly insightful, since for matrix $\mA$ that has orthogonal rows, the right-hand side is zero, predicting convergence of Kaczmarz method in a single cycle, as expected. Moreover, the provided bound agrees with existing empirical observations. When the matrix rows are near-orthogonal, the inner products  $\innp{\hat{\va}_i, \hat{\va}_j}$ are small, and the stated corollary predicts fast convergence as long as $\sigma_r^2(\hat{\mA}) >> \sum_{i < j}\innp{\hat{\va}_i, \hat{\va}_j}^2.$ When the inner products $\innp{\hat{\va}_i, \hat{\va}_j}$ are close to one, meaning that the rows are near-parallel, the contraction factor $\frac{\sum_{1\le i < j \le m} \innp{\hat{\va}_i, \hat{\va}_j}^2}{\sum_{1\le i < j \le m} \innp{\hat{\va}_i, \hat{\va}_j}^2 + \sigma_{r}^2(\hat{\mA})}$ is of the order $1 - O(\frac{\sigma_r^2}{m^2 + \sigma_r^2}) \approx 1- \frac{1}{m^2}$ for $\sigma_r = \Theta(1),$ predicting slow convergence observed in practice (see also the examples in \cite{sun2021worst}).  Finally, to obtain a fully quantitative bound, we provide an upper estimate of $\norm{\mL - \mI_m}_\op$ in the following lemma. The second term improves upon a similar bound obtained in \cite{Oswald_2015}; see \Cref{sec:related-work}. 
\begin{lemma}\label{lem:L-I-bnd}
    Given a matrix $\hat{\mA} \in \sR^{m \times n}$ with $\rank(\mA) = r,$ let $\mL$ be defined as in \eqref{eq:L-and-M-defs}, where, as before $\mA = \mU \mSigma \mV^\top$ is the singular value decomposition of $\mA,$ $\hat{\mA}$ is its row-normalized form, and $\hat{\va}_i$ are the rows of $\hat{\mA}$. Then:
    \[
    \norm{\mL - \mI_m}_\op^2 \leq \min\Big\{\sum_{1\le i < j \le m} \innp{\hat{\va}_i, \hat{\va}_j}^2, \, \Big(\frac{\lceil\log_2(r)\rceil}{2} + 1\Big)^2\sigma_1^4(\hat{\mA})\Big\}.
    \]
\end{lemma}
\begin{proof}
    The first bound in the minimum is immediate, as
    $$\|\mL - \mI_m \|_\op^2 \le \|\mL - \mI_m \|_{F}^2 = \sum_{1\le i < j \le m} \innp{\hat{\va}_i, \hat{\va}_j}^2.$$

    For the second bound in the minimum, we provide a recursive argument that splits the strictly lower triangle in $\mL- \mI_m$ into sets of non-overlapping submatrices. In particular, observe that all non-zero entries of $\mL - \mI_m$ are under the main diagonal, with entry $(i, j)$ with $i > j$ equal to $\innp{\hat{\va}_i, \hat{\va}_j}.$ At the same entry position, the matrix $\hat{\mA}{\hat{\mA}}^\top$ takes the same value. Observe further that
    \begin{equation}\label{eq:half-singular-val}
        \norm{\hat{\mA}\hat{\mA}^\top - \frac{\sigma_1^2(\hat{\mA})}{2}\mI_m}_\op \leq \frac{\sigma_1^2(\hat{\mA})}{2}. 
    \end{equation}
    Let $c = \lfloor \frac{m + 1}{2}\rfloor,$ and consider the rectangular submatrix of $\mL - \mI_m,$ obtained by selecting rows $c+1, \dots, m$ and columns $1, 2, \dots, c$. By triangle inequality, $\norm{\mL - \mI_m}_\op$ is at most the sum of the operator norms of the selected rectangular matrix and of the matrix that remains when the selected entries are subtracted from $\mL - \mI_m$. The former is at most  $\frac{\sigma_1^2(\hat{\mA})}{2},$ since the rectangular matrix is also a submatrix of $\hat{\mA}\hat{\mA}^\top - \frac{\sigma_1^2(\hat{\mA})}{2}\mI_m$. The latter matrix now has two nonzero triangles under the main diagonal, with non-overlapping row/column indices (the first one is spanned by rows $1, 2, \dots, c$ and columns $1, \dots, c$, the second is spanned by rows $c + 1, \dots, m$ and columns $c+1, \dots, m$). We now recursively apply the described partitioning to the induced triangles. The key insight is that all inscribed rectangles at the same level of recursion have non-overlapping row/columns and are submatrices of $\hat{\mA}\hat{\mA}^\top - \frac{\sigma_1^2(\hat{\mA})}{2}\mI_m;$ thus each recursive level adds $\frac{\sigma_1^2(\hat{\mA})}{2}.$ We stop the recursion when all induced triangles under the diagonal have at most $q := \big\lceil \frac{1}{r} \sum_{i=1}^r \sigma_i^2(\hat{\mA}) \big\rceil$ rows and at most $q$ columns. Because 
    \[
    m = \norm{\hat{\mA}}_F^2 = \sum_{i=1}^r \sigma_i^2(\hat{\mA}) = r \frac{1}{r} \sum_{i=1}^r \sigma_i^2(\hat{\mA}),
    \]
    the recursion does not reach depth higher than $\big\lceil\log_2\big( \frac{m}{q}\big)\big\rceil \leq \lceil\log_2(r)\rceil.$ As a result, the combined contribution to $\norm{\mL - \mI_m}_\op$ of all the rectangles is at most $\lceil\log_2(r)\rceil \frac{\sigma_1^2(\hat{\mA})}{2}.$ 

    It remains to bound the operator norm of the remaining small triangles (with sides of size at most $q = \big\lceil \frac{1}{r} \sum_{i=1}^r \sigma_i^2(\hat{\mA}) \big\rceil$) under the diagonal. Notice that these triangles are non-overlapping, so their combined operator norm is bounded by the maximum operator norm over all these triangles. For each small triangle we can bound the operator norm by the Frobenius norm; since each non-zero element is of the form $\innp{\hat{\va}_i, \hat{\va}_j}$ (with $\innp{\hat{\va}_i, \hat{\va}_j}^2 \leq 1$) and there are at most $\frac{q(q-1)}{2}$ nonzero elements (since, recall, we stopped partitioning as soon as all triangle sides were at most $q$), the operator norm of the contribution of all triangles is at most $\sqrt{\frac{q(q-1)}{2}} \leq q-1 \leq \frac{1}{r} \sum_{i=1}^r \sigma_i^2(\hat{\mA}).$

    Combining with the bound on the contribution from all the rectangles, we get $\norm{\mL - \mI_m}_\op \leq \lceil\log_2(r)\rceil \frac{\sigma_1^2(\hat{\mA})}{2} + \frac{1}{r} \sum_{i=1}^r \sigma_i^2(\hat{\mA}) \leq (\lceil\log_2(r)\rceil/2 + 1)\sigma_1^2(\hat{\mA}),$ as  $\frac{1}{r} \sum_{i=1}^r \sigma_i^2(\hat{\mA}) \leq \sigma_1^2(\hat{\mA}).$ 
\end{proof}
We are now ready to state and prove the complete, quantitative bound. 
\begin{corollary}[Row-correlation bound]\label{cor:row-correlation-bound}
    Given a consistent linear system \eqref{eq:main-prob} with $\rank(\mA) = r$, the squared residual of classical Kaczmarz method \eqref{eq:Kaczmarz-update-full}, $\| \mA \vx_k - \vb\|^2 = \|\mA(\vx_k- \vx^*)\|^2$, converges geometrically. Specifically, for $k \geq 1,$
    \begin{align*}
    \| \mA \vx_k - \vb\|^2 \le \min\bigg\{&  \frac{\max_{i \in [m]}\norm{\va_i}^2\norm{\mL - \mI_m}_\op^2\norm{\mA\mP_m}_\op^2}{\norm{\mA\mP_m}_\op^2 + \max_{i \in [m]}\norm{\va_i}^2\norm{\mL - \mI_m}_\op^2}\bigg(\frac{\|\mL - \mI_m \|_\op^2}{\|\mL - \mI_m \|_\op^2 + \sigma_r^2(\hat{\mA})}\bigg)^{k-1},\\
    & \norm{\mA\mP_m}_\op^2 \bigg(\frac{\|\mL - \mI_m \|_\op^2}{\|\mL - \mI_m \|_\op^2 + \sigma_r^2(\hat{\mA})}\bigg)^{k}\bigg\}\|\vx_0 -\vx^* \|^2, \; \text{ where }\\
    \norm{\mA \mP_m}_\op^2 &\leq \min\Big\{\sum_{i=1}^m \norm{\va_i}^2\big(1 - \innp{\hat{\va}_i, \hat{\va}_m}^2\big), \sigma_1^2(\mA)\Big\}, \\
    \|\mL - \mI_m \|_\op^2 &\le \min\Big\{\sum_{1\le i < j \le m} \innp{\hat{\va}_i, \hat{\va}_j}^2, \, (1 + \lceil\log_2(r)\rceil/2)^2\sigma_1^4(\hat{\mA})\Big\}.
    \end{align*}
\end{corollary}
\begin{proof}
    By \Cref{thm:matrix-power-bound} and submultiplicativity of the spectral norm,
    \begin{align*}
        \| \mA \vx_k - \vb\|^2
        &\le \norm{\mSigma \mM}_\op^2 \|\mM^{k-1}\|_\op^2 \|\vx_0-\vx^*\|^2 \\
        &\le \norm{\mSigma \mM}_\op^2 \|\mM\|_\op^{2(k-1)} \|\vx_0-\vx^*\|^2.
    \end{align*}
    It remains to use \Cref{lem:spectrum-of-SigmaM} and \Cref{lem:spectrum-of-M} to bound $\norm{\mSigma \mM}_\op^2, \|\mM\|_\op^{2}$, respectively, and \Cref{lem:L-I-bnd} to bound $\norm{\mL - \mI_m}_\op.$  
    Since $f_a(t) = \frac{t}{t+a}$ is monotonically increasing in $t$ for $a > 0$, substituting these bounds gives the claimed result.
\end{proof}

\begin{remark}
    We briefly remark that the relaxation for $\|\mL - \mI_m \|_\op^2$ from \Cref{lem:L-I-bnd}, similar to \Cref{cor:row-correlation-bound}, also leads to sublinear rate bounds expressed in terms of row correlations and/or matrix spectrum and rank, by substituting this inequality into the bounds from \Cref{thm:matrix-power-bound}. 
\end{remark}

\begin{remark}\label{rem:convergence-of-the-iterates}
    The quantitative bounds on $\norm{\mM}_\op$ further imply iterate convergence bounds, as, by \Cref{thm:matrix-power-bound}, $\norm{\vx_k - \vx^*} \leq \norm{\mM^k}_\op \norm{\vx_0 - \vx^*} \leq \norm{\mM}_\op^k \norm{\vx_0 - \vx^*}$, and we can similarly use \Cref{lem:spectrum-of-M} to obtain an upper estimate of $\norm{\mM}_\op$ to obtain a fully quantifiable convergence rate estimate.   
\end{remark}

\subsection{Discussion of Tightness of Provided Convergence Bounds}\label{sec:tightness}

The provided convergence bounds \eqref{eq:tight-contraction}, \eqref{eq:tight-telescoping} in \Cref{thm:matrix-power-bound} are instance-tight, up to the choice of the initialization, but are not fully interpretable in terms of quantities such as the correlations among the matrix vectors, which have long been known to affect the convergence of Kaczmarz method. On the other hand, the more explicit bounds dependent on row correlations, the matrix rank, and its smallest nonzero singular value, provided in \Cref{thm:matrix-power-bound} and in \Cref{cor:row-correlation-bound} are based on suitable relaxations of the provided tight bounds. We now discuss what those bounds predict in simple special cases and their worst-case tightness, based on simple examples. 

\paragraph{Tightness of convergence bounds for parallel row vectors.} When the row vectors $\va_i$ are all parallel to each other (so $|\innp{\hat{\va}_i, \hat{\va}_j}| = 1$ for all $i, j \in [m]$), the provided linear convergence bounds all correctly predict the method's convergence in a single cycle. The reason is that \Cref{lem:spectrum-of-SigmaM} bounds the prefactor $\norm{\mSigma \mM}_\op^2$ proportional to $\norm{\mA\mP_m}_\op^2 \leq \sum_{i=1}^m \norm{\va_i}^2\big(1 - \innp{\hat{\va}_i, \hat{\va}_m}^2\big) = 0,$ so for any $k \geq 1,$ the stated bound on the residual is zero. Additionally, the bound from \Cref{thm:matrix-power-bound} also directly provides a zero on the right-hand side through the second term in the minimum, as in this case $r:= \rank(\mA) = 1.$

\paragraph{Tightness of convergence bounds for orthogonal row vectors.} Kaczmarz method is also known to converge in a single cycle when the rows $\va_i$ are mutually orthogonal (i.e., $\innp{\hat{\va}_i, \hat{\va}_j} = 0$ for all $i \neq j,$ $i, j \in [m]$). In this case, the contraction factor $\frac{\sum_{1\le i < j \le m} \innp{\hat{\va}_i, \hat{\va}_j}^2}{\sum_{1\le i < j \le m} \innp{\hat{\va}_i, \hat{\va}_j}^2 + \sigma_r^2(\hat{\mA})}$ from \Cref{cor:row-correlation-bound} equals zero, and the stated bound correctly predicts convergence in a single cycle. In fact, in this case, all stated bounds involving $\norm{\mL - \mI_m}_\op$ correctly predict convergence in a single cycle, as, based on the bound from \Cref{cor:row-correlation-bound}, $\norm{\mL - \mI_m}_\op = 0.$

\paragraph{Tightness of the convergence bounds for neither parallel nor orthogonal vectors.} First, to evaluate tightness of the convergence bound prefactors obtained in \Cref{lem:spectrum-of-SigmaM}, we consider the following $2\times 2$ example: given $c \in (0, 1),$ define
\begin{equation}\label{eq:2x2-example}
    \mA_c = \begin{pmatrix}
        1 & 0\\
        c & \sqrt{1-c^2}
    \end{pmatrix},\quad  \vb = \vzero,\quad  \vx_0 = (0, 1)^\top.
\end{equation}
Observe that for this example, $\mA = \hat{\mA}$ and the rows are neither parallel nor orthogonal.  
Kaczmarz updates in this case lead to $\vx_k = c^{2k-1}(-\sqrt{1-c^2}, c)^\top$ and $\norm{\mA_c \vx_k}^2 = (1-c^2)c^{4k-2}$, for all $k \geq 1$. Further, a direct calculation gives $\norm{\mL - \mI_2}_\op^2 = c^2,$ $\norm{\mA \mP_2}_\op^2 = 1- c^2,$ $\norm{\mM}_\op^2 = c^2,$ and $\norm{\mSigma\mM}_\op^2 = c^2(1-c^2).$ Thus, again by a direct calculation, both bounds from \Cref{lem:spectrum-of-SigmaM} equal $\norm{\mSigma\mM}_\op^2 = c^2(1-c^2)$ and are thus tight:
\begin{equation}\notag
    (1-c^2)c^2 = \norm{\mA\mP_2}_\op^2 \norm{\mM}_\op^2 = \frac{\norm{\mL - \mI_2}_\op^2\norm{\mA\mP_2}_\op^2}{\norm{\mL - \mI_2}_\op^2 + \norm{\mA\mP_2}_\op^2},
\end{equation}
As a result, the linear rate bound in \Cref{cor:row-correlation-bound} is tight (i.e., holds with equality) for $k = 1$. 

Moreover, for $k \geq 2,$ since $\norm{\mSigma\mM}_\op^2\norm{\mM^{k-1}}_\op^2 = (1- c^2)c^{4k-4}$ and, as already noted, $\norm{\mA_c \vx_k}^2 = (1-c^2)c^{4k-2}$, the first bound in \Cref{thm:matrix-power-bound} is within a factor $1/c^2$ of the true value of the squared residual for $k \geq 2.$ This factor can be made arbitrarily close to one by letting $c \to 1$ (but not equal to one). The sublinear rate bound for the average residual is also within a constant factor of the true value in this case. In particular, the exact value of the averaged residual is $\frac{1}{k}\sum_{i=1}^k \norm{\mA_c \vx_i}^2 = \frac{c^2}{1+c^2}\frac{1-c^{4k}}{k}$, whereas the bound from \Cref{thm:matrix-power-bound} leads to $\frac{1}{k}\sum_{i=1}^k \norm{\mA_c \vx_i}^2 \leq \frac{c^2}{k}.$ This predicted bound is within a factor $\frac{1 + c^2}{1 - c^{4k}}$ of the true value and can be made arbitrarily close to one for any $k \geq 1$  by letting $c \to 0$ (but not equal to zero).

Finally, on the provided example, \Cref{cor:row-correlation-bound} leads to the bound $\norm{\mA_c\vx_k}^2 \leq c^2(1-c^2)\big(\frac{c^2}{c^2 + 1 - c}\big)^{k-1} = c^2(1-c^2)\big(1-q_{\mathrm{bnd}}\big)^{k-1},$ $q_{\mathrm{bnd}} = \frac{1-c}{c^2 + 1 - c},$ whereas $\norm{\mA_c \vx_k}^2 = (1-c^2)c^{4k-2} = c^2(1-c^2)c^{4(k-1)} = c^2(1-c^2)(1 - q_{\mathrm{true}})^{k-1},$ $q_{\mathrm{true}} = 1 - c^4$. Thus, the predicted rate of convergence $q_{\mathrm{bnd}}$ is within a factor $ \frac{q_{\mathrm{true}}}{q_{\mathrm{bnd}}} = (1 + c + c^2 + c^3)(c^2 + 1 - c) = 1 + c^2 + c^3 + c^5$ of the true rate $q_{\mathrm{true}}$ for any $k \geq 1,$ and, for any value of $c \in (0, 1),$ $q_{\mathrm{bnd}} \leq q_{\mathrm{true}} \leq 4 q_\mathrm{bnd}.$ Letting $c \to 0$ (but not equal to zero), this bound can be made arbitrarily close to one.   

Finally, fix $k \geq 1$ and set $c^2 = 1 - 1/(2k).$ Then $\norm{\mA\vx_k}^2 = \frac{1}{2k}(1 - 1/(2k))^{2k-1} \geq 1/(2ek).$ In particular, the order of decay $1/k$ cannot be improved in the worst case. On the same example, the bound from \Cref{cor:row-correlation-bound} is at most $c^2(1-c^2)(\frac{c^2}{c^2 + 1 - c})^{k-1} \leq \frac{1}{2k},$ thus within a factor $e$ of the true residual. In other words, the complete  bound (involving both the prefactor and $(1- q_{\mathrm{bnd}})^{k-1}$) is tight up to an absolute constant ($\leq e \approx 2.718$) in the worst case. 

\paragraph{The role of matrix conditioning and rank.} Finally, since in the provided $2\times 2$ example tightness of the sublinear rate bound was argued by making rows close to being parallel ($c \to 1$), we provide a $3 \times 3$ example for which the rows are far from being parallel, yet convergence is arbitrarily slow and the sublinear rate bound is tight up to a universal constant in the worst-case, for any fixed value of $k$. In particular, for $\eta \in (0, 1/12],$ consider the family of matrices
\begin{equation}
    \mA_\eta := \begin{pmatrix}
        \sqrt{1-\eta} & 0 & \sqrt{\eta}\\
        -\frac{1}{2}\sqrt{1-\eta} & \frac{\sqrt{3}}{2}\sqrt{1-\eta} &\sqrt{\eta}\\
        -\frac{1}{2}\sqrt{1-\eta} & -\frac{\sqrt{3}}{2}\sqrt{1-\eta} &\sqrt{\eta}
    \end{pmatrix}.
\end{equation}
For this matrix family, all rows have unit norm, and all row pairs $i \neq j$ have $\innp{\va_i, \va_j} = c := \frac{3\eta - 1}{2} \in (-1/2, - 3/8],$ thus are far from being either parallel or orthogonal, and all entries of $\mL$ under the main diagonal are equal to $c$. Additionally, the squared singular values of $\mA_\eta$ are $\sigma_1^2(\mA_\eta) = \sigma_2^2(\mA_\eta) = \frac{3}{2}(1-\eta),$ $\sigma_3^2(\mA_\eta) = 3\eta,$ so the minimum singular value becomes arbitrarily small as $\eta$ is reduced towards zero. 

We argued in \eqref{eq:residual-identity} (recalling here that for the considered example, $\hat{\mA}_\eta = \mA_\eta$) that ${\mA}_\eta(\vx_{k+1} - \vx^*) = -(\mL - \mI)^\top \mL^{-1}{\mA}_\eta(\vx_k - \vx^*).$ A direct calculation shows that the matrix $-(\mL - \mI)^\top \mL^{-1}$, which determines the evolution of the residual, is
\begin{equation}
    -(\mL - \mI)^\top \mL^{-1} = \begin{pmatrix}
        2c^2 - c^3 & c^2 - c & -c\\
        c^2 - c^3 & c^2 & -c\\
        0 & 0 & 0
    \end{pmatrix}.
\end{equation}
This matrix has two nonzero eigenvalues, both of which are roots of $f(z) = z^2 - (3c^2 - c^3)z + c^3,$ and one of which is positive, the other negative (since $c < 0$). Let the positive root be denoted by $\lambda_\eta.$ We claim that $\lambda_\eta \geq 1- 6\eta.$ Indeed, this follows from $1 - 6 \eta \geq \frac{1}{2} > 0$ (as $0 < \eta \leq \frac{1}{12}$) and $f(1-6\eta) = - \frac{9}{4}\eta^2(5-30\eta + 9\eta^2)<0.$  

Now choose $\vx_0$ so that $\mA_{\eta}(\vx_0 - \vx^*)$ is an eigenvector of $-(\mL - \mI)^\top \mL^{-1}$ corresponding to the eigenvalue $\lambda_\eta \geq 1- 6\eta,$ where we choose $\vb \in \range(\mA_\eta), \vb \neq \mA_\eta \vx_0$ (for instance, $\vb = \vzero,$ in which case $\vx^* = \vzero$). Then $\norm{\mA(\vx_k - \vx^*)}^2 = \lambda_\eta^{2k}\norm{\mA(\vx_0 - \vx^*)}^2 \geq 3\eta \lambda_\eta^{2k}\norm{\vx_0 - \vx^*}^2$, as $\sigma_3^2(\mA_\eta) = 3\eta.$ Fix any $k \geq 1$ and set $\eta = \frac{1}{12k}.$ Then $\norm{\mA(\vx_k - \vx^*)}^2 \geq \frac{1}{4k}\big(1 - \frac{1}{2k}\big)^{2k}\norm{\vx_0 - \vx^*}^2 \geq \frac{\norm{\vx_0 - \vx^*}^2}{16k}.$ On the other hand, $\norm{\mL - \mI}_\op^2 = \frac{3 + \sqrt{5}}{2}c^2,$ so the last-iterate sublinear rate bound from \Cref{thm:matrix-power-bound} gives $\norm{\mA(\vx_k - \vx^*)}^2 \leq \frac{2\norm{\mL - \mI}_\op^2\norm{\vx_0 - \vx^*}^2}{k} \leq \frac{(3 + \sqrt{5})\norm{\vx_0 - \vx^*}^2}{4k}$. In other words, this bound is within a constant factor of the true value of the residual, and the worst-case sublinear rate $1/k$ cannot be improved even when the rows of the input matrix are far from being either parallel or orthogonal to each other. Another way to view this example is through the lens of the minimum singular value $\sigma_3^2(\mA_\eta) = 3\eta,$ which is the main culprit for the linear rate predicted by the bounds in \Cref{thm:matrix-power-bound} and \Cref{cor:row-correlation-bound} becoming excessively slow and, as a result, the residual value is within a constant factor of the bound predicted by the sublinear rate $1/k$.    

Further, the convergence rate $q_{\mathrm{bnd}}$ predicted by \Cref{cor:row-correlation-bound} is $q_{\mathrm{bnd}} = \frac{3\eta}{3c^2 + 3\eta} = \frac{4\eta}{1 - 2\eta + 9 \eta^2} \geq 4\eta,$ while the true convergence rate is $q_{\mathrm{true}} = 1 - \lambda_\eta^2 \leq  12\eta.$ In other words, predicted scaling with $\sigma_r^2(\mA_\eta) = 3\eta$ is necessary, and the bound on $q_{\mathrm{bnd}}$ is within a factor of three of the smallest value that can be obtained in any worst-case linear rate bound.  

The contraction factor in our bounds in \Cref{cor:row-correlation-bound} is also bounded by $1 - \frac{\sigma_r^2(\hat{\mA})}{(\lceil\log_2(r)\rceil/{2} + 1)^2 {\sigma_1^4(\hat{\mA})} + \sigma_r^2(\hat{\mA})}$, so our bound predicts  $q_\mathrm{bnd} \geq \frac{\sigma_r^2(\hat{\mA})}{(\lceil\log_2(r)\rceil/{2} + 1)^2 {\sigma_1^4(\hat{\mA})} + \sigma_r^2(\hat{\mA})}$. This bound is only dependent on the matrix spectrum and rank. First, the $1/\log^2(r)$ scaling with the rank is known to be necessary in the worst case, at least when it comes to the single-cycle contraction factor. This follows from the example in \cite[Section 3]{oswald1994convergence}; see also the discussion in \cite{sun2021worst,oswald2017random}. 

Regarding the dependence on the singular values via $\sigma_1^4(\hat{\mA})$ and $\sigma_r^2(\hat{\mA}),$ the necessity of the convergence rate scaling with $\frac{\sigma_r^2(\hat{\mA})}{\sigma_1^4(\hat{\mA})}$ follows from \cite[Example 2 in the arXiv version, arXiv:1604.07130]{sun2021worst} (although the authors used this example to argue about cases when randomized Kaczmarz method is faster than the classical, cyclic variant). In particular, in their example, the matrix $\mA$ is defined row-wise as $\va_j = (\cos(2\pi j/m), \sin(2\pi j/m))^\top,$ $j \in [m].$ The rows have norm equal to one, so $\hat{\mA} = \mA.$ It is possible to argue from this example that $q_{\mathrm{true}} = \Theta(\frac{\sigma_2^2(\mA)}{\sigma_1^4(\mA)}),$ where $\frac{\sigma_2^2(\mA)}{\sigma_1^4(\mA)} = \frac{2}{m}$, thus the scaling with $\frac{\sigma_2^2(\mA)}{\sigma_1^4(\mA)}$ is necessary in the worst case, and, in particular, it cannot be replaced by the inverse squared spectral condition number equal to $\frac{\sigma_2^2(\mA)}{\sigma_1^2(\mA)}$. Here we provide a more general example that allows the ratio $\frac{\sigma_2^2(\mA)}{\sigma_1^4(\mA)}$ to become arbitrarily small while retaining constant-factor tight scaling of the convergence rate with this ratio. As a result, we give a family of examples for which Kaczmarz method can be arbitrarily slow.

To construct the needed family of instances, we use the ``back-and-forth'' rotation instance from \cite[proof of Theorem 11]{evron2022catastrophic}.\footnote{\cite{evron2022catastrophic} uses this instance for a different purpose; namely to construct a worst-case sublinear rate example.} 
Fix an even integer $m \geq 8$ and an ordering vector $\vo := (1,3, \dots, m-1, m, m-2, m-4, \dots, 2)^\top$. Fix $t \in (0, \pi/m]$. Define the matrix $\mA_t$ via its row vectors
\begin{equation}\label{eq:conditioning-example}
\va_j = (\cos(\vo^{(j)}t), \sin(\vo^{(j)}t))^\top, \; j \in [m].
\end{equation}
Observe that the rows have the norms equal to one, so $\mA_t = \hat{\mA_t}.$ A direct calculation for $\mA_t^\top \mA_t$ based on \eqref{eq:conditioning-example} and the observation that the rows $\va_j$ correspond to rotations in the complex plane (the resulting trigonometric sum corresponds to $\sum_{j=1}^m e^{2\i j t} = e^{\i(m+1)t}\frac{\sin(m t)}{\sin(t)},$ where $\i$ denotes the imaginary unit) leads to
\begin{equation}\label{eq:rotation-example}
    \mA_t^\top\mA_t = \frac{m}{2}\mI_2 + \frac{\sin(mt)}{2\sin(t)}\begin{pmatrix}
        \cos((m+1)t) & \sin((m+1)t)\\
        \sin((m+1)t) & -\cos((m+1)t)
    \end{pmatrix}.
\end{equation}
The rightmost $2\times 2$ matrix in \eqref{eq:rotation-example} has eigenvalues $-1$ and $1.$ As a result, 
the eigenvalues of $\mA_t^\top\mA_t$ (and so the squared singular values of $\mA_t$) are $\sigma_1^2(\mA_t) = \frac{m}{2} + \frac{\sin(mt)}{2\sin(t)}$ and $\sigma_2^2(\mA_t) = \frac{m}{2} - \frac{\sin(mt)}{2\sin(t)}$. In particular, $\lim_{t \downarrow 0} \frac{\sigma_2^2(\mA_t)}{\sigma_1^4(\mA_t)} = 0$ and $\frac{\sigma_2^2(\mA_t)}{\sigma_1^4(\mA_t)} = \frac{2}{m}$ for $t = \pi/m$. Thus $\frac{\sigma_2^2(\mA_t)}{\sigma_1^4(\mA_t)} \in (0, 2/m].$ More generally, since $\sigma_1^2(\mA_t) \in [m/2, m]$ and we can write $\sigma_2^2(\mA_t) = \sum_{j=1}^m \sin^2\big(\big(j - \frac{m+1}{2}\big)t\big),$ using the standard inequality $\frac{4u^2}{\pi^2}\leq \sin^2(u) \leq u^2$ (which holds for $|u| \leq \pi/2$), we can bound $\frac{m(m^2 - 1)}{3\pi^2}t^2 \leq \sigma_2^2(\mA_t) \leq \frac{m(m^2 - 1)}{12}t^2$. As a result,
\begin{equation}
   \frac{m^2-1}{3\pi^2m}t^2 \leq \frac{\sigma_2^2(\mA_t)}{\sigma_1^4(\mA_t)} \leq \frac{m^2 - 1}{3m}t^2,
\end{equation}
establishing $\frac{\sigma_2^2(\mA_t)}{\sigma_1^4(\mA_t)} = \Theta(mt^2)$ along the entire range $t \in (0, \pi/m]$. 

Let us now characterize the true convergence behavior of Kaczmarz method. Define $\vw_j := (-\sin(jt), \cos(jt))^\top$. The $j^{\mathrm{th}}$ within-cycle update of Kaczmarz method \eqref{eq:Kaczmarz-update-within-cycle} is the orthogonal projection to the subspace of $\sR^n$ orthogonal to $\va_j,$ which can be expressed using the projection operator $\mP_j = \mI_2 - \va_j \va_j^\top = \vw_{\vo^{(j)}}\vw_{\vo^{(j)}}^\top$. Further, $\vw_{\vo^{(i)}}^\top \vw_{\vo^{(j)}} = \cos((\vo^{(i)}-\vo^{(j)})t).$ Thus, a full cycle update corresponds to
\begin{align*}
    \vx_{k+1} &= \mP_m \mP_{m-1} \dots \mP_1 \vx_k \\
    &= \vw_{\vo^{(m)}}\vw_{\vo^{(m)}}^\top \vw_{\vo^{(m-1)}}\vw_{\vo^{(m-1)}}^\top \dots \vw_{\vo^{(1)}}\vw_{\vo^{(1)}}^\top \vx_k\\
    &= \cos^{m-2}(2t)\cos(t)\vw_{2}\vw_{1}^\top \vx_k.
\end{align*}
Let $\vb = \vzero$ (and so $\vx^* = \vzero$), and initialize Kaczmarz method at $\vx_0 = \vw_{\vo^{(m)}} = \vw_2$. Then, since $\vw_2^\top\vw_1 = \vw_{\vo^{(m)}}^\top \vw_{\vo^{(1)}} = \cos(t)$, we get $\vx_{k} = (\cos^{m-2}(2t)\cos^2(t))^k\vx_0$, and, as a result, \[\norm{\vx_{k} - \vx^*}^2 = (\cos^{m-2}(2t)\cos^2(t))^{2k}\norm{\vx_0 - \vx^*}^2.\] 
To convert this into a statement involving convergence rate, we bound below the contraction factor $\cos^{2(m-2)}(2t)\cos^4(t)$ using $\cos^2(u) = 1 -\sin^2(u),$ for all $u,$ and the product inequality $\prod_{i=1}^N (1-u_i) \geq 1 - \sum_{i=1}^N u_i,$ which holds for all $u_i \in [0, 1]$ and all $N \geq 1.$ In particular,
\begin{align*}
    q_{\mathrm{true}} \leq 1- \cos^{2(m-2)}(2t)\cos^4(t) &= 1 - (1 - \sin^2(2t))^{m-2}(1 - \sin^2(t))^2\\
    &\leq  (m-2)\sin^2(2t) + 2\sin^2(t)\\
    &\leq (4m - 6)t^2,
\end{align*}
where in the last inequality we used $|\sin(u)| \leq u.$ Observe that  $q_{\mathrm{true}} \leq (4m - 6)t^2 = \Theta\big(\frac{\sigma_2^2(\mA_t)}{\sigma_1^4(\mA_t)}\big).$ Thus, combining with our bound $q_{\mathrm{bnd}} \geq \frac{\sigma_2^2(\mA_t)}{(9/4)\sigma_1^4(\mA_t) + \sigma_2^2(\mA_t)} \geq \frac{2}{5}\frac{\sigma_2^2(\mA_t)}{\sigma_1^4(\mA_t)}$ (see \Cref{cor:row-correlation-bound}), the convergence rate of Kaczmarz method is proportional to $\frac{\sigma_2^2(\hat{\mA})}{\sigma_1^4(\hat{\mA})}$ in the worst case, and this dependence cannot be improved. Since our example allows for this ratio to be arbitrarily small, iterates of Kaczmarz method can converge at an arbitrarily slow rate in the worst case. 

\subsection{When is Classical Cyclic Ordering Faster than Randomized?}

Finally, we discuss what the obtained convergence results tell us about the instances on which classical Kaczmarz method studied here is faster than its randomized variant introduced in \cite{Vershynin_2007}. We note that the complementary examples on which the classical version performs worse are already known from \cite{sun2021worst}. 

To provide a clean convergence comparison of Kaczmarz method under randomized and cyclic ordering, consider row-normalized matrices $\mA = \hat{\mA},$ so the uniform and the importance-weighted sampling in randomized Kaczmarz method are the same. For a fair comparison, one full cycle of classical Kaczmarz method should be compared to $m$ updates of the randomized variant. To disambiguate the bounds for the cyclic and randomized updates, let $\vx_k^{\mathrm{CK}}$ be the iterates of the classical Kaczmarz method \eqref{eq:Kaczmarz-update-full} and let $\vx_k^{\mathrm{RK}}$ be the iterates of the randomized update, corresponding to \eqref{eq:Kaczmarz-update-within-cycle} with rows sampled uniformly at random. We compare the bounds in terms of the linear  convergence rates for the iterates; additional advantages of the cyclic update for the squared residual may arise from better prefactors (see \Cref{cor:row-correlation-bound}) and sublinear convergence bounds from sufficiently small $k$ (see \Cref{thm:matrix-power-bound} and the examples in \Cref{sec:tightness}). 

Our bound from \Cref{cor:row-correlation-bound} (see \Cref{rem:convergence-of-the-iterates}) implies that
\begin{equation}\label{eq:CK-bound}
    \norm{\vx_k^{\mathrm{CK}} - \vx^*}^2 \leq \Big( \frac{\norm{\mL - \mI_m}_\op^2}{\norm{\mL - \mI_m}_\op^2 + \sigma_r^2(\mA)}\Big)^k \norm{\vx_0 - \vx^*}^2. 
\end{equation}
To argue about cases when the cyclic updates are faster, we need a lower bound on the convergence rate of the randomized version. For this, we use the tight characterization of Kaczmarz method from \cite{steinerberger2021randomized}. In particular, for $\vx_0 - \vx^* = \vv_r$ (slow singular-direction initialization), we have $\E[\vx_N^{\mathrm{RK}} - \vx^*] = \big(1 - \frac{\sigma_r^2(\mA)}{m}\big)^N \vv_r,$ for all $N \geq 1.$ By Jensen's inequality, for all $N \geq 1,$
\begin{equation}\label{eq:RK-bound}
    \E[\norm{\vx_{N}^{\mathrm{RK}} - \vx^*}^2] \geq \Big(1 - \frac{\sigma_r^2(\mA)}{m}\Big)^{2N} \norm{\vx_0 - \vx^*}^2. 
\end{equation}
Comparing \eqref{eq:CK-bound} and \eqref{eq:RK-bound} with $N = mk,$ we conclude that a sufficient condition for classical Kaczmarz method to outperform the randomized version under the considered initialization is that
\begin{equation}\label{eq:CK-vs-RK}
    \frac{\norm{\mL - \mI_m}_\op^2}{\norm{\mL - \mI_m}_\op^2 + \sigma_r^2(\mA)} < \Big(1 - \frac{\sigma_r^2(\mA)}{m}\Big)^{2m}\; \Leftrightarrow \; \norm{\mL - \mI_m}_\op^2 < \frac{\sigma_r^2(\mA)(1 - {\sigma_r^2(\mA)}/{m})^{2m}}{1 - (1 - {\sigma_r^2(\mA)}/{m})^{2m}}.
\end{equation}
From \Cref{lem:L-I-bnd}, $\norm{\mL - \mI_m}_\op^2 \leq \min\Big\{\sum_{1\le i < j \le m} \innp{{\va}_i, {\va}_j}^2, \, \Big(\frac{\lceil\log_2(r)\rceil}{2} + 1\Big)^2\sigma_1^4({\mA})\Big\},$ so this condition holds for near-orthogonal rows. In particular, for any $\sigma_r^2(\mA)$ that can be treated as a universal positive constant independent of $m$, $(1 - {\sigma_r^2(\mA)}/{m})^{2m}$ is a universal constant inside the interval $(0, 1)$. Because $\sum_{1\le i < j \le m} \innp{{\va}_i, {\va}_j}^2$ approaches zero as rows become closer to orthogonal (as $\innp{{\va}_i, {\va}_j} = \cos(\angle({\va}_i, {\va}_j))$; notice also that $\sigma_r^2(\mA)$ approaches one as rows become closer to orthogonal), condition \eqref{eq:CK-vs-RK} is satisfied for a range of near-orthogonal linear systems. 

Alternatively, consider the case $\sigma_r^2(\mA) < 1/2$. Since by Bernoulli's inequality $\big(1 - \frac{\sigma_r^2(\mA)}{m}\big)^{2m} \geq 1 - 2\sigma_r^2(\mA)$, a sufficient condition for cyclic updates to be faster is that 
\[
\norm{\mL - \mI_m}_\op^2 < \frac{1}{2} - \sigma_r^2(\mA). 
\]

Finally, the $2\times 2$ example from \eqref{eq:2x2-example} is a specific instance family on which the cyclic update always converges faster than the randomized update for initialization perpendicular to the second row, $\vx_0 = (- \sqrt{1-c^2}, c)^\top$ (with all other parameters the same as in \eqref{eq:2x2-example}). On this family of instances, $\norm{\vx_k^{\mathrm{CK}}}^2 = c^{4k}\norm{\vx_0}^2$ while $\E\norm{\vx_{2k}^{\mathrm{RK}}}^2 = \big(\frac{1 + c^2}{2}\big)^{2k}\norm{\vx_0}^2$. Because $c^4 < \big(\frac{1 + c^2}{2}\big)^{2},$ cyclic update is faster for every instance in this family (i.e., for every $c \in (0, 1)$). Notably, this family includes instances where the rows may not be nearly orthogonal nor nearly parallel.  
\section{Conclusion}

We presented new convergence bounds for the classical Kaczmarz method for solving consistent linear systems and argued about their tightness. It would be interesting to obtain comparable bounds for possibly inconsistent systems and various generalizations of the method, including block variants and more general projections onto convex sets. 

\section*{Acknowledgements}

This work was supported in part by the NSF CAREER Award CCF-2440563 and by the NSF MFAI Award DMS-2502282. 

\paragraph{AI disclosure.} The main ideas for this work came from the authors, based on J.\ Diakonikolas' past work on related cyclic methods \cite{cai2024tighter,cai2026near,song2023cyclic} and the insights regarding convergence of randomized Kaczmarz methods \cite{steinerberger2021randomized} presented by S.\ Steinerberger at U.\ Washington in January 2026. The core arguments (\Cref{lem:dynamics-of-ek}, \Cref{lem:spectrum-of-M}, and the initial version of \Cref{thm:matrix-power-bound} that contained row correlation-based contraction factor) were fully developed by the authors, without any use of AI. Using these initial results, the authors then used OpenAI's GPT-6 (Astra) to audit the results and search for possibly tighter bounds and examples demonstrating tightness. These were carried out in a highly interactive flow. The authors further used AI to aid in the literature search; all the references and related statements were verified by the authors. All the arguments in the paper, as well as all the prose, were written by the authors, then polished (primarily correcting typos) by the AI. The authors take the sole responsibility for the entire content of the paper.

\bibliographystyle{abbrv}
\bibliography{refs}

\end{document}